\documentclass{article} 
\usepackage{geometry}
\usepackage{amsmath}
\usepackage{amssymb}
\usepackage{amsthm}
\usepackage{graphicx}
\usepackage{subcaption}
\usepackage{float}
\usepackage{multirow}
\usepackage[normalem]{ulem}
\usepackage[utf8]{inputenc}
\usepackage{csquotes}
\usepackage[english]{babel}
\usepackage{setspace}
\usepackage{algorithm}
\usepackage{enumitem}
\usepackage{mathpazo}
\usepackage{mathtools}
\usepackage{bbm}
\usepackage{booktabs}
\usepackage{centernot}

\usepackage{cancel}

\usepackage{tikz}
\usetikzlibrary{shapes,decorations.pathreplacing,arrows,calc,arrows.meta,fit,positioning}
\tikzset{
    -Latex,auto,node distance =1 cm and 1 cm,semithick,
    state/.style ={circle, draw, minimum width = 1 cm, inner sep=0pt},
    point/.style = {circle, draw, inner sep=0.04cm,fill,node contents={}},
    bidirected/.style={Latex-Latex,dashed},
    el/.style = {inner sep=2pt, align=left, sloped}
}

\useunder{\uline}{\ul}{}
\usepackage[authoryear,round,sort&compress]{natbib}
\usepackage[colorlinks=true, linkcolor=blue]{hyperref}

\newtheorem{theorem}{Theorem}[section]
\newtheorem{lemma}[theorem]{Lemma}
\newtheorem{corollary}[theorem]{Corollary}

\newtheorem*{theorem*}{Theorem}
\newtheorem*{lemma*}{Lemma}
\newtheorem*{corollary*}{Corollary}
\newtheorem*{proposition*}{Proposition}
\newtheorem*{conjecture*}{Conjecture}

\theoremstyle{definition}
\newtheorem{definition}{Definition}
\newtheorem*{definition*}{Definition}
\theoremstyle{definition}

\theoremstyle{definition}
\newtheorem*{example*}{Example}
\theoremstyle{definition}

\theoremstyle{definition}
\newtheorem*{assumption*}{Assumption}
\theoremstyle{definition}

\theoremstyle{remark}
\newtheorem{remark}{Remark}[section]
\theoremstyle{remark}
\newtheorem*{remark*}{Remark}

\DeclareMathOperator*{\argmin}{arg\,min} 

\DeclareMathOperator{\supp}{supp}

\newcommand{\E}{\mathbb{E}}
\newcommand{\prob}{\mathbb{P}}

\newcommand{\AND}{\text{ and }}
\newcommand{\given}{\,|\,}
\newcommand{\indep}{\perp \!\!\! \perp }
\newcommand{\T}{^\top}

\newcommand{\bss}{\text{BSS}}
\newcommand{\bssa}{\text{BSSu}}

\newcommand{\dd}{d}
\newcommand{\sps}{s}
\newcommand{\ubsps}{\overline{s}}
\newcommand{\betam}{\beta_{\min}}
\newcommand{\betaspace}{\Theta}
\newcommand{\Sigmaspace}{\Omega}
\newcommand{\mineig}{\omega}
\newcommand{\supps}{\mathcal{S}}
\newcommand{\suppspace}{\supps_{\dd,\sps}}
\newcommand{\suppspaceub}{\supps_{\dd}^{\ubsps}}
\newcommand{\penalt}{\tau}

\newcommand{\trusupp}{S_{*}} 
\newcommand{\estsupp}{\widehat{S}}
\newcommand{\mclass}{\mathcal{M}} 
\newcommand{\ubmclass}{\overline{\mathcal{M}}} 
\newcommand{\signalone}{\Delta}

\newcommand{\nbeta}{\beta}
\newcommand{\neps}{\epsilon}

\usepackage{authblk}

\begin{document}

\title{Optimality and computational barriers in\\variable selection under dependence}

\author[]{Ming Gao}
\author[]{Bryon Aragam}
\affil[]{\emph{University of Chicago}}

\date{}

\maketitle

{\let\thefootnote\relax\footnote{Contact: \texttt{\{minggao,bryon\}@chicagobooth.edu}}}

\begin{abstract}%
    We study the optimal sample complexity of variable selection in linear regression under general design covariance, and show that subset selection is optimal while under standard complexity assumptions, efficient algorithms for this problem do not exist.
    Specifically, we analyze the variable selection problem and provide the optimal sample complexity with exact dependence on the problem parameters for both known and unknown sparsity settings.
    Moreover, we establish a sample complexity lower bound for any efficient estimator, highlighting a gap between the statistical efficiency achievable by combinatorial algorithms (such as subset selection) compared to efficient algorithms (such as those based on convex programming).
    The proofs rely on a finite-sample analysis of an information criterion estimator, which may be of independent interest.
    Our results emphasize the optimal position of subset selection, the critical role played by restricted eigenvalues, and characterize the statistical-computational trade-off in high-dimensional variable selection.
\end{abstract}

\section{Introduction}\label{sec:intro}
Variable selection is a classical problem in statistical learning theory. It aims to select the most effective subset of variables for predicting a target variable. 
The application of variable selection is ubiquitous, including feature selection in machine learning \citep{guyon2003introduction}, structure learning in graphical models \citep{maathuis2018handbook}, covariate adjustment in causal inference \citep{guo2022confounder}, and scientific research in biology \citep{heinze2018variable}.
Variable selection has drawn more attention in the high-dimensional era, and many computationally efficient algorithms have been proposed and studied \citep{tibshirani1996regression,fan2008sure}. 

We consider the variable selection task for linear regression with Gaussian noise:
\begin{align}\label{eq:lm}
    Y = X \T \nbeta + \neps\,, 
    \qquad X\sim\mathcal{N}(0,\Sigma) \,,
    \qquad \neps\sim\mathcal{N}(0,\sigma^2) \,.
\end{align}
In this setup, variable selection is also known as support recovery of $\nbeta$ vector.
We take the view of minimax optimality.
In this thread, existing work primarily considers standard design \citep{aeron2010information,reeves2019all}, i.e. the covariance $\Sigma = I_\dd$ where $\dd$ is the dimension, and conclude the optimal sample complexity.
While standard design is reasonable in certain experimental settings, it is more practical to study general design covariance beyond $I_\dd$ in modern applications. In particular, the potential presence of strong dependence among variables (such as in graphical models) imposes difficulties for existing computationally efficient methods, making it of special theoretical interest.
Although finite sample analysis in general design has been explored in prior work \citep{wainwright2009information,shen2012likelihood,shen2013constrained}, a gap remains between the obtained upper and lower bounds with respect to the problem parameters of $\nbeta$ and $\Sigma$. 
On the other hand, the computational cost of variable selection is known to be prohibitively high \citep{natarajan1995sparse}. 
Therefore, simultaneous statistical and computational optimality are unresolved.  This poses the natural question of whether or not there exists a simultaneously computationally efficient and sample optimal algorithm under general dependence.

To approach this, it is necessary to first establish matching upper and lower bounds on the sample complexity while setting aside computational considerations. Only then 
can we attempt characterize the (potential) statistical-computational trade-off in variable selection.
We focus on problems with general design covariance, which exhibits strong dependence between variables. Hence, naive approaches based on thresholding or marginal independence testing between response and covariates typically fails, and $\ell_1$-penalty based methods or other computationally efficient alternatives potentially require additional assumptions to perform, further highlighting the computational challenges of this problem.
Moreover, existing studies on sample complexity typically assume the exact knowledge of sparsity level $\sps$, which is the number of nonzero entries of $\nbeta$. While convenient to simplify the analysis, it sidesteps important technical challenges that arise when the sparsity is unknown and requires special treatment. 
Naturally, it is more practical and realistic to assume only an upper bound $\ubsps\ge \sps$ is provided. Under this setup, while consistency results are available for information-criterion type estimators, the finite sample analysis is still an open line of research. Therefore, the question of potential extension of optimality from known sparsity case to this more general setting has yet to be addressed.

\subsection{Contributions}\label{sec:intro:contri}
Our main contributions can be summarized as follows:
\begin{itemize}
    \item For variable selection in Gaussian linear model under general design, we derive the optimal sample complexity with precise dependence on the problem parameters:
    \begin{align}\label{eq:samcom}
        \Theta\bigg(\frac{\log \dd }{\mineig \betam^2/\sigma^2} \vee \log \binom{\dd-\sps}{\sps} \bigg)
    \end{align}
    where $\dd,\sps,\betam,\mineig,\sigma^2$ are dimension, sparsity level, nonzero entry lower bound of $\nbeta$, minimum eigenvalue lower bound of $\Sigma$, and noise variance. We show that best subset selection (\bss{}) achieves the optimality.
    \item We extend this optimality result to the unknown sparsity case where only an upper bound on the sparsity $\sps\le\ubsps$ is known. We show the optimality of \bss{} with an additive penalty in this more general setting, and the optimal sample complexity is given by replacing $\sps$ with $\ubsps$ in~\eqref{eq:samcom}.
    \item We provide a sample complexity lower bound for any polynomial-time support estimator with a gap between the optimality by a factor of restricted eigenvalue. This demonstrates a sample-computation trade-off in variable selection problem.
\end{itemize}

\subsection{Related work}\label{sec:intro:related}
As a fundamental and long-lasting problem, the existing work on variable selection is rich in nature. We intend to review the most relevant work to our focus of minimax optimality, particularly under the lens of general design matrices, without diving into every aspect of variable selection.
The concept of variable selection traces back to the foundational work of ANOVA \citep{fisher1936design,fisher1936use,fisher1970statistical}. Numerous textbook methods have been developed since then, including the stepwise regression \citep{efroymson1960multiple,miller1984selection,draper1998applied}, best subset selection \citep{hocking1967selection}, various types of information criteria \citep{mallows2000some,akaike1974new,schwarz1978estimating,konishi1996generalised}, and cross-validation \citep{stone1974cross}. The large sample properties of these methods, e.g. asymptotic efficiency and (in)consistency, have been obtained in different asymptotic regimes \citep{nishii1988maximum,shao1993linear,shao1997asymptotic}.
The modern era of high-dimensional data has seen substantial interest in $\ell_1$-based methods. Important representatives include Lasso \citep{tibshirani1996regression} and Orthogonal Matching Pursuit \citep{tropp2007signal}. These methods achieve variable selection consistency under conditions like irrepresentability or mutual incoherence \citep{zhao2006model,wainwright2009sharp,zhang2011sparse,cai2011orthogonal}. 
Extensions to them include thresholding the Lasso-type estimates \citep{meinshausen2009lasso}, replacing the $\ell_1$ penalty with other nonconvex choices \citep{fan2001variable,loh2017support,zhang2010nearly}, and multi-stage methods combining estimation power and thresholding thereby \citep{ndaoud2020optimal,wang2020bridge,wasserman2009high,ji2012ups}.
While these approaches are efficient, they typically require specific assumptions on the covariance, e.g. bounded norms or (restricted) eigenvalues of the design matrices. While in this work, we aim for the direction of arbitrary, strong dependence among the covariates, where these assumptions may fail in general.

We focus on analyzing the sample complexity of exact recovery under general design.
Optimality for the standard design, which considers $\Sigma=I_\dd$ and is related to compressed sensing \citep{candes2007dantzig,akccakaya2009shannon}, has been derived as $\Theta(\log\dd / \betam^2 \vee \log \binom{\dd-\sps}{\sps})$ \citep{rad2011nearly,fletcher2009necessary,aeron2010information}, where $\betam$ is the assumed lower bound of nonzero entries of $\nbeta$ (cf. Section~\ref{sec:prelim}). The optimality can be achieved by efficient methods \citep{wainwright2009sharp,ndaoud2020optimal}.
By contrast, for general design, the efficient techniques no longer apply due to the dependence between variables, and the sample complexity analysis is more complicated. Existing bounds are present in \citet{wainwright2009information,wang2010informationl,shen2012likelihood,shen2013constrained} along with analysis of best subset selection (\bss{}), but matching bounds with exact dependence on the problem parameters is not established yet. 
Moving beyond the known sparsity level, while consistency of information criterion-based methods is well-documented \citep{nishii1988maximum}, their finite-sample behavior, especially under general design, requires further study.

Exact subset selection is known to be computationally hard \citep{natarajan1995sparse,foster2015variable}. Over the years, much progress has been made toward solving the \bss{} programming more efficiently. Notable developments include mixed integer programming \citep{bertsimas2016best}, coordinate descent \citep{hazimeh2020fast}, and binary convex reformulation \citep{bertsimas2020sparse}. In particular, \citet{zhu2020polynomial} employs a sequencing-and-splicing technique to solve the programming in polynomial time, albeit imposing a sparse restricted condition (SRC).
On the other hand along with these computational advances, gaps in sample complexity performance between efficient methods and theoretical optimality have been established for many statistical problems \citep{kunisky2019notes,bandeira2022franz,moitra2023precise}, e.g. sparse PCA \citep{berthet2013optimal}, low-rank matrix problems \citep{oymak2015simultaneously,ma2015computational}, and Gaussian mixture models \citep{diakonikolas2017statistical}. 
In the context of linear model, \citet{zhang2014lower} has demonstrated a gap between the minimax prediction risk and the performance achievable by any polynomial-time algorithms by a factor of restricted eigenvalue \citep{raskutti2010restricted}. While in this work, we focus on the variable selection aspect of linear model.

\vspace{-1em}
\subsection{Notation}\label{sec:intro:notation}
For any nonnegative integer $\dd$, let $[\dd]:=\{1,\ldots,\dd\}$. For $\dd\ge 1$, throughout the paper, $S$ and $T$ are subsets of $[\dd]$ with $|S|$ being the cardinality.
Denote set of all possible subsets of $[\dd]$ with size (sparsity) $\sps $ to be $\suppspace:=\{S\subseteq[\dd]:|S|=\sps\}$. Further denote all subsets with size bounded by $\ubsps$ to be $\suppspaceub:= \cup_{\sps=0}^{\ubsps}\suppspace=\{S\subseteq[\dd]:|S|\le \ubsps\}$.
For a vector $x$, write the 2-norm to be $\|x\|=(\sum_{j}x_j^2)^{1/2}$, and the support to be $\supp(x) = \{j:x_j\ne 0\}$.
For a matrix $A$, write the operator 2-norm to be $\|A\|=\|A\|_{\textup{op}}=\sup_{\|x\|=1}\|Ax\|$, and the largest and smallest eigenvalues to be $\lambda_{\max}(A)$ and $\lambda_{\min}(A)$.
Let $x_S$ be the sub-vector of $x$ with coordinates indexed by $S$. Analogously, for matrix $A$, let $A_{S}$ be the sub-matrix with columns indexed by set $S$, and $A_{TS}$ to be the sub-matrix with rows and columns indexed by $T$ and $S$.
Let $\mathbb{R}_+,\mathbb{Z}_+,\mathbb{S}^d_{++}$ be positive numbers, positive integers, and positive definite matrices.
For a covariance matrix $\Sigma \in \mathbb{S}^d_{++}$, denote the conditional covariance matrix of the variables $S$ given the variables $T$ by $\Sigma_{S\given T}:= \Sigma_{S} - \Sigma_{ST}\Sigma_{TT}^{-1}\Sigma_{TS}$.
Let $\mathbf{1}_m,\mathbf{0}_m$ be all one's and all zero's vector of dimension $m$.
With some abuse of notation, we use $(X,Y)$ for both the random variables and the data matrix ($\mathbb{R}^{n\times \dd}\otimes \mathbb{R}^{n}$) interchangeably.
We denote by $\Pi_S:=X_S(X_S\T X_S)^{-1}X_S\T$ and $\Pi_S^\perp:=I_n-\Pi_S$ the projection matrices onto and out of the column subspace of $X_S$.
Finally, we say $a\lesssim b$ and $a\gtrsim b$ if $a\le Cb$ and $a\ge cb$ for some positive constants $C$ and $c$, and $a\asymp b$ if both $a\lesssim b$ and $a\gtrsim b$ hold.
$a\vee b$ and $a\wedge b$ are the maximum and minimum between two numbers $a$ and $b$, and $\lfloor a\rfloor$ is the largest integer small than or equal to $a$.

\section{Preliminaries}\label{sec:prelim}
We consider the usual linear model with Gaussian noise as in~\eqref{eq:lm}, copied below for reference:
\begin{align*}
    Y = X \T \nbeta + \neps\,, 
    \qquad X\sim\mathcal{N}(\mathbf{0}_\dd,\Sigma) \,,
    \qquad \neps\sim\mathcal{N}(0,\sigma^2)\,, \qquad X\indep \neps.
\end{align*}
The coefficient vector $\nbeta$ is sparse in the sense that $\|\nbeta\|_0 = \sps \le \ubsps$. 
We assume exact knowledge of $\sps$ in Section~\ref{sec:opt:bss}, then relax to unknown sparsity setting in Section~\ref{sec:opt:bss:unknown} where only an upper bound $\ubsps\ge\sps$ is provided.
In this work, we put the most basic assumptions on $\nbeta$ and $\Sigma$ and study the optimal dependence on the signal strength implied by these assumptions. Specifically, we consider the following parameter spaces:
\begin{align*}  
    \betaspace_{\dd,\sps}(\betam) & :=  \Big\{\nbeta\in\mathbb{R}^{\dd}: \|\nbeta\|_0 = \sps, \min_{j\in\supp(\nbeta)}|\nbeta_j|\ge \betam > 0\Big\},  \\
    \Sigmaspace_{\dd,\sps}(\mineig) & := \Big\{\Sigma\in\mathbb{S}^\dd_{++}: \min_{S,T\in\suppspace}\lambda_{\min}(\Sigma_{S\setminus T\given T})\ge \mineig > 0\Big\} \,. 
\end{align*}
The space $\betaspace_{\dd,\sps}(\betam)$ consists of all sparse vectors with exact $\sps$ many nonzero entries, and each of them is bounded away from zero by at least $\betam$, which measures the signal for recovery and is commonly assumed in literature \citep{van2011adaptive}. 
The covariance matrix space $\Sigmaspace_{\dd,\sps}(\mineig)$ imposes a lower bound on the minimum eigenvalue of the conditional covariance over subsets of size $\sps$, which essentially requires that the variance in $X_S$ to not be fully explained by $X_T$, otherwise $S$ and $T$ would be indistinguishable. 
For the special case of standard design ($\Sigma=I_\dd$), $\min_{S,T\in\suppspace}\lambda_{\min}(\Sigma_{S\setminus T\given T})=1$ for all $S,T\in\suppspace$.
We can extend these parameter spaces to the unknown sparsity setting as follows:
\begin{align*}  
    \betaspace_{\dd}^{\ubsps}(\betam) & :=  \Big\{\nbeta\in\mathbb{R}^{\dd}: \|\nbeta\|_0 \le \ubsps, \min_{j\in\supp(\nbeta)}|\nbeta_j|\ge \betam > 0\Big\},  \\
    \Sigmaspace_{\dd}^{\ubsps}(\mineig) & := \Big\{\Sigma\in\mathbb{S}^\dd_{++}: \min_{S,T\in\suppspaceub,S\not\subseteq T}\lambda_{\min}(\Sigma_{S\setminus T\given T})\ge \mineig > 0\Big\} \,.
\end{align*}
We enlarge the $\nbeta$ vector space by relaxing the exact sparsity to being upper bounded by $\ubsps$. For $\Sigma$ space, we relax the sizes of the candidate subsets while requiring one is not fully contained in the other. In addition, we denote
\begin{align}\label{eq:mclass}
\begin{aligned}
    \mclass & := \Big\{(\nbeta,\Sigma,\sigma^2):\nbeta\in\betaspace_{\dd,\sps}(\betam),\Sigma\in  \Sigmaspace_{\dd,\sps}(\mineig) \Big\} \\ 
    \ubmclass & := \Big\{(\nbeta,\Sigma,\sigma^2):\nbeta\in\betaspace_{\dd}^{\ubsps}(\betam),\Sigma\in  \Sigmaspace_{\dd}^{\ubsps}(\mineig) \Big\} \,,
\end{aligned}
\end{align}
to be the model classes of known and unknown sparsity settings. 
We often suppress the dependence on $(\dd,\sps,\ubsps,\betam,\mineig,\sigma^2)$ to avoid notation clutter.

We aim to estimate the support of $\nbeta$, denoted as $\trusupp:=\supp(\nbeta)$. A support estimator $\estsupp$ is a measurable function of i.i.d. observations $(X,Y)$ to the power set of $[\dd]$, i.e. $\estsupp(X,Y)\subseteq [\dd]$. We study the sufficient and necessary conditions on the sample size $n$ in terms of the problem parameters $\dd,\sps,\sigma^2,\betam,\mineig$ 
such that the error probability of exact recovery of $\trusupp$ is upper bounded by any small constant $\delta>0$ uniformly over $\mclass$ (or $\ubmclass$):
\begin{align}\label{eq:prelim:guarantee}
    \sup_{(\nbeta,\Sigma,\sigma^2)\in\mclass} \prob(\estsupp\ne \trusupp) \le \delta \,.
\end{align}
We derive upper and lower bounds on sample size $n$ such that above holds. When the derived bounds are matched up to problem-independent constants and logarithmic factor of sparsity $\log\sps$, we refer to them as the optimal sample complexity. Note that we do not suppress the factor of $\log\dd$. Any support estimator is called optimal when it satisfies~\eqref{eq:prelim:guarantee} with the optimal sample complexity.
In addition, as commonly imposed in literature, we assume $\sps\le\ubsps\le\dd/2$ to simplify results.

\begin{remark}
    Crucially, unlike the common assumptions in the literature for general design, which impose various types of bounded ``eigenvalues'', we do \emph{not} treat either problem parameter $\betam$, $\mineig$ or $\sigma^2$ as fixed constants; instead, we are interested in studying how selection performance in terms of sample complexity depends on these quantities.
\end{remark}

Finally, one crucial quantity in our result is the restricted eigenvalue (RE) of the design matrix, which appears extensively in previous work on $\ell_1$-penalized estimators and related methods.
It relaxes the typical dependence of the estimation error on the minimum eigenvalue of the empirical covariance (which vanishes when $n > \dd$). We recall the definition here:
\begin{definition}\label{defn:poly:RE}
    The RE constant of a design matrix $X\in\mathbb{R}^{n\times \dd}$ is
    \begin{align*}
        \gamma(X) := \min_{S\in\suppspace}\min_{\|\theta_{S^c}\|_1 \le 3\|\theta_{S}\|_1}\frac{\|X\theta\|^2 / n}{\|\theta\|^2} \,.
    \end{align*}
\end{definition}

\section{Optimality with known sparsity}\label{sec:opt:bss}
In this and the next section, we derive the optimal sample complexity for variable selection in linear models under general design with known and unknown sparsity. To achieve this, we derive a new lower bound to match existing upper bounds.
For completeness, we begin by reviewing one such upper bound before detailing our lower bound.
First recall the definition of \bss{}:
\begin{align*}
    \estsupp^{\bss{}} := \argmin_{S\in \suppspace} \|\Pi_S^\perp Y\|^2\,,
\end{align*}
which estimates the true support by minimizing the residual variance of $Y$ over all possible supports $\suppspace$. Recall that $\Pi_S^\perp$ is the projection matrix out of the column space of $X_S$.

For the setting with the knowledge of $\sps = \|\nbeta\|_0$, 
we start with defining the generic signal to distinguish $\trusupp$ from any other alternatives $T$ to be:
\begin{align}\label{eq:bss:signal}
    \signalone : = \min_{T\in\suppspace\setminus\{\trusupp\}}\frac{1}{|\trusupp\setminus T|} \frac{\nbeta_{\trusupp\setminus T}\T \Sigma_{\trusupp\setminus T\given T}\nbeta_{\trusupp\setminus T}}{\sigma^2}\,. 
\end{align}
We can find the intuition of this signal by looking at the numerator, which is the expected residual variance contributed by $\trusupp$ but not fully captured by the alternative support $T$, and depends on the difference set $(\trusupp\setminus T)$. The denominator is the noise variance of $\neps$.
Since \bss{} minimizes the residual variance for variable selection, \eqref{eq:bss:signal} measures how much signal-noise-ratio that \bss{} can exploit. 
Then we have the following generic lemma on statistical guarantee of \bss{}:
\begin{lemma}
\label{lem:opt:bss:ub:generic}
Assuming $\sps\le \dd/2$, for any $(\nbeta,\Sigma,\sigma^2) \in \mclass$, let $\trusupp=\supp(\nbeta)$, given $n$ i.i.d. samples from $P_{\nbeta,\Sigma,\sigma^2}$, if the sample size
\begin{align}\label{eq:opt:bss:ub:generic}
    n  &\gtrsim  \max_{\ell\in[\sps]}\frac{\log\binom{\dd-\sps}{\ell} + \log(1/\delta)}{\big(\ell\signalone\big) \wedge 1 } \,,
\end{align}
then $\prob_{\nbeta,\Sigma,\sigma^2}(\estsupp^{\bss{}}=\trusupp)\ge 1-\delta$.
\end{lemma}
The proof is postponed to Appendix~\ref{app:opt:bss:ub}. We emphasize the key idea behind the proof lies in the scaling factor $|\trusupp\setminus T|$ in~\eqref{eq:bss:signal}: the signal to distinguish $\trusupp$ and $T$ is actually proportional to $|\trusupp\setminus T|$. This means when $T$ deviates from $\trusupp$ a lot (by the number of missing true covariates), it is easier to tell them apart.
At the same time, the total number of alternatives $T$ to $\trusupp$ also grows with their difference $|\trusupp\setminus T|$. Therefore, these two effects cancel each other, leading to the desired sample complexity $\log (\dd-\sps) / \signalone \vee \log\binom{\dd-\sps}{\sps}$.
By applying Lemma~\ref{lem:opt:bss:ub:generic} to $\mclass$, we obtain an upper bound on the sample complexity for this model class, which was first described in \citet{wainwright2009information}:
\begin{theorem}[\citealp{wainwright2009information}, Theorem~1]
\label{thm:opt:bss:ub}
Assuming $\sps\le \dd/2$, for any $(\nbeta,\Sigma,\sigma^2) \in \mclass$, given $n$ i.i.d. samples from $P_{\nbeta,\Sigma,\sigma^2}$, if the sample size
\begin{align}\label{eq:opt:bss:ub}
    n  &\gtrsim  \max\bigg\{ \frac{\log\big(\dd-\sps\big) + \log(1/\delta)}{\betam^2\mineig/\sigma^2} , \log \binom{\dd-\sps }{\sps} + \log(1/\delta)\bigg\} \,,
\end{align}
then $\prob_{\nbeta,\Sigma,\sigma^2}(\estsupp^{\bss{}}=\trusupp)\ge 1-\delta$.
\end{theorem}
\begin{proof}
The proof is given by the following chain of inequalities: For any $(\nbeta,\Sigma,\sigma^2)\in \mclass$ and any $T\in\suppspace\setminus\{\trusupp\}$, we have
\begin{align*}
    \nbeta_{\trusupp\setminus T}\T\Sigma_{\trusupp\setminus T\given T}\nbeta_{\trusupp\setminus T} 
     \ge \|\nbeta_{\trusupp\setminus T}\|^2\lambda_{\min}(\Sigma_{\trusupp\setminus T\given T}) 
    \ge |\trusupp\setminus T|\betam^2\lambda_{\min}(\Sigma_{\trusupp\setminus T\given T}) 
     \ge |\trusupp\setminus T|\betam^2\mineig \,,
\end{align*}
which yields $\signalone\ge \betam^2\mineig / \sigma^2$ and completes the proof.
\end{proof}
Theorem~\ref{thm:opt:bss:ub} establishes the upper bound $\log (\dd-\sps) / (\betam^2\mineig/\sigma^2) + \log\binom{\dd-\sps}{\sps}$ for variable selection and provides characterization of the dependence on $(\nbeta,\Sigma)$ via $\betam$ and $\mineig$ separately. 
$\betaspace_{\dd,\sps}(\betam)$ requires each variable in $\trusupp$ has large enough effect on $Y$, and $\Sigmaspace_{\dd,\sps}(\mineig)$ demands for any two distinct supports $S$ and $T$, the variables therein cannot fully explain each other. Both parameters are indispensable for the uniform consistency of successful support recovery.

Now we switch gears to obtain lower bounds for the risk over $\mclass$ to match the ones in Theorem~\ref{thm:opt:bss:ub}.
The lower bound provided in Theorem~2 of \citet{wainwright2009information} is
\begin{align*}
    n \gtrsim \max\bigg\{\frac{\log \binom{\dd}{\sps}}{\mineig_{bu}\betam^2 / \sigma^2} , \frac{\log(\dd-\sps)}{\mineig_{ave}\betam^2 / \sigma^2}\bigg\} \,,
\end{align*}
where
\begin{align*}
    & \mineig_{bu}:= \E_{S}[\min_{z_{S}\in\mathbb{R}^\sps,|z_j|\ge 1, \forall j}z_S\T \Sigma_{SS}z_S] \\
    & \mineig_{ave}:= \E_S[\min_{t\in S}\min_{z_u:u\in\{t\}\cup S^c,|z_u|\ge 1/\sqrt{2}} \sum_{u,v\in \{t\}\cup S^c}(\Sigma_{uu}z_u^2 + \Sigma_{vv}z_v^2- 2\Sigma_{uv}z_uz_v)] \,,
\end{align*}
and the expectation is taken over $S\sim Unif(\suppspace)$.
However, this lower bound is derived using different set of parameters $\mineig_{bu}$ and $\mineig_{ave}$ rather than $\mineig$, which do not exactly match the upper bound in Theorem~\ref{thm:opt:bss:ub}.
This is mainly because $\mineig_{ave}\ge \mineig$ and $\mineig_{bu}/\sps \ge \mineig$ in general. For the special case of standard design $\Sigma=I_\dd$, we have $\mineig_{ave}= \mineig = 1$ and $\mineig_{bu}=\sps$. 
Nonetheless, the difference between them can be large: A simple $2\times 2$ covariance $\Sigma=[\begin{smallmatrix} 1 & b \\ b & 1+b^2 \end{smallmatrix}]$ with $\sps=1$ and some moderate $b> 0$ leads to $\mineig_{ave} \wedge \mineig_{bu} \ge 1$ while $\mineig = 1 / (1+b^2)$.
Therefore, here we want to derive a lower bound that precisely matches the upper bound with the same set of parameters, i.e. $\betam,\,\sigma^2, \AND \mineig$.

The first lower bound construction aims to match the first term in upper bound~\eqref{eq:opt:bss:ub}, i.e. to characterize the dependence on the problem parameters $\betam,\sigma^2$, and especially $\mineig$, which is the variance that cannot be explained by variables outside of $\trusupp$. 
The construction is built upon a equi-correlation matrix where any pair of variables are equally correlated with correlation specified by $\mineig$, i.e. a rank one perturbation of identity matrix:
\begin{align*}
    \Sigma_{\mineig} : = \mineig I_\dd + (1-\mineig) \mathbf{1}_\dd\mathbf{1}_\dd\T \,.
\end{align*}
When $\mineig$ is close to zero, $\Sigma_\mineig$ is dense in the off-diagonal entries and exhibits strong dependence among all variables $X$.
In particular, we have $\mineig_{ave} = \mineig$ in this construction.
Fixing this choice of $\Sigma$, we consider $\sps$ many ensembles by enumerating all possible candidate supports according to their difference to the truth $|\trusupp\setminus T|$, and then combine the lower bounds obtained into one final lower bound. The proof is in Appendix~\ref{app:opt:bss:lb1}.
\begin{theorem}\label{thm:opt:bss:lb1}
Assuming $\mineig<1$, given $n$ i.i.d. samples from $P_{\nbeta,\Sigma,\sigma^2}$ with $(\nbeta,\Sigma,\sigma^2)\in \mclass$, if the sample size is bounded as
\begin{align}\label{eq:opt:bss:lb1}
    n & \le \frac{1-2\delta}{2}\times \max_{\ell\in[\sps]} \frac{\log \binom{\dd-\sps}{\ell}}{\ell\betam^2\mineig / \sigma^2}  \nonumber \\
    & \asymp \frac{1-2\delta}{2}\times \frac{\log(\dd-\sps)}{\betam^2\mineig/ \sigma^2} \,,
\end{align}
then for any estimator $\estsupp$ for $\trusupp=\supp(\nbeta)$,
\begin{align*}
    \inf_{\estsupp}\sup_{(\nbeta,\Sigma,\sigma^2) \in \mclass} \prob_{\nbeta,\Sigma,\sigma^2}(\estsupp\ne \trusupp) \ge \delta - \frac{\log 2}{\log\binom{\dd-\sps}{\sps}} \,.
\end{align*}
\end{theorem}
To match the second term in the upper bound~\eqref{eq:opt:bss:ub} that only depends on dimension parameters, we invoke Theorem~1 of \citet{wang2010informationl} in Theorem~\ref{thm:opt:bss:lb3} below.
We will assume $\betam^2/\sigma^2$ is upper bounded by some constant, but it should be presumably small.
The construction in \citet{wang2010informationl} fixes standard design $\Sigma=I_\dd$ and $\nbeta_{\trusupp}=\betam\mathbf{1}_\sps$, then considers the ensemble of all possible supports $\suppspace$.
The analysis further relies on the fact that the differential entropy of a continuous random variable is maximized by a Gaussian distribution with matched variance.
\begin{theorem}[\citealp{wang2010informationl}, Theorem~1]\label{thm:opt:bss:lb3}
    Given $n$ i.i.d. samples from $P_{\nbeta,\Sigma,\sigma^2}$ with $\nbeta\in\betaspace_{\dd,\sps}(\betam)$, $\Sigma=I_\dd$. If the sample size is bounded as 
    \begin{align}\label{eq:opt:bss:lb3}
        n\le 2(1-\delta)\times \frac{\log\binom{\dd}{\sps}-1}{\log(1 + \sps\betam^2/\sigma^2)}\,,
    \end{align}
    then for any estimator $\estsupp$ for $\trusupp=\supp(\nbeta)$, 
    $$\inf_{\estsupp}\sup_{\nbeta\in\betaspace_{\dd,\sps}(\betam)} \prob_{\nbeta,I_\dd,\sigma^2}(\estsupp\ne \trusupp) \ge \delta\,.$$
\end{theorem}
The standard design $\Sigma=I_\dd$ considered in Theorem~\ref{thm:opt:bss:lb3} is a special case of general design of $\mclass$, thus the lower bound obtained also applies to $\mclass$.
Combined with Theorem~\ref{thm:opt:bss:lb1}, the two lower bounds~\eqref{eq:opt:bss:lb1}-\eqref{eq:opt:bss:lb3} match the upper bound~\eqref{eq:opt:bss:ub}, and we verify the folklore and conclude \bss{} is optimal for variable selection problem under general design with knowledge of the sparsity level.

\section{Optimality with unknown sparsity}\label{sec:opt:bss:unknown}
Having concluded the optimal sample complexity in the known sparsity case, we now extend this result to the setting where $\sps$ is unknown but has a known upper bound $\ubsps$. Formally, we will derive new sample complexity upper bound for $\ubmclass$ defined in~\eqref{eq:mclass}.
The minimax estimator is achieved by modifying \bss{} with an additive penalty depending on the dimensionality and the model parameters---similar to information criteria such as BIC---to help us target the truth $\trusupp$.
The finite sample analysis for this estimator and its optimality is new to the best of our knowledge. 

Given a tuning parameter $\penalt$, define an estimator $\bssa{}$ (where ``u'' stands for ``unknown'') by:
\begin{align}
\label{eq:opt:bss:bssa}
    \estsupp^{\bssa{}} = \argmin_{S\in\suppspaceub} \frac{\|\Pi^\perp_S Y\|^2}{n-|S|} + |S|\penalt \,.
\end{align}
The first term is the residual variance objective of \bss{} and the second term is a penalty term.
\begin{theorem}\label{thm:opt:bss:ub:unknown}
Assuming $\ubsps\le \dd/2$, for any $(\nbeta,\Sigma,\sigma^2)\in \ubmclass$, let $\trusupp=\supp(\nbeta)$, given $n$ i.i.d. samples from $P_{\nbeta,\Sigma,\sigma^2}$, choose $\penalt = \frac{1}{4}\mineig\betam^2$,
if the sample size
\begin{align}\label{eq:opt:bss:ub:unknown}
    n  &\gtrsim  \max\bigg\{ \frac{\log \dd + \log(1/\delta)}{\betam^2\mineig/\sigma^2} , \log \binom{\dd}{\ubsps} + \log(1/\delta)\bigg\}\,,
\end{align}
then $\prob_{\nbeta,\Sigma,\sigma^2}(\estsupp^{\bssa{}}=\trusupp)\ge 1-\delta$.
\end{theorem}
Since the support space expands from $\suppspace$ to $\suppspaceub$, the upper bound now has dependence on $\ubsps$. Implicitly, $\mineig$ in this setting should be perceived as smaller compared to the known sparsity setting. Because given a fixed $\mineig$, the covariance matrix spaces have a nested relationship $\Sigmaspace_{\dd}^{\ubsps}(\mineig)\subseteq\Sigmaspace_{\dd,\sps}(\mineig)$. 
Overall, compared to Theorem~\ref{thm:opt:bss:ub}, variable selection with unknown sparsity is harder than the known case, which is intuitively due to the lack of exact knowledge of $\sps$.

\begin{remark}\label{rmk:opt:bss:aicbic}
The delicate choice of $\tau=\mineig\betam^2/4$ in Theorem~\ref{thm:opt:bss:ub:unknown} serves as a balance to correctly identify the alternative $T$ without over-penalizing $\trusupp$.
We emphasize our analysis is finite-sample as opposed to classic MLE theory relying on asymptotics \citep[e.g.][]{nishii1988maximum}, thus provides an alternative understanding of regularization choice.
In light of~\eqref{eq:opt:bss:bssa}, we can draw connections to two traditional model selection methods: AIC \citep{akaike1974new} and BIC \citep{schwarz1978estimating}, given by 
\begin{align*}
    \estsupp^{\textup{AIC}}  = \argmin_{S\in\suppspaceub} \frac{\|\Pi^\perp_S Y\|^2}{n} + |S|\frac{2}{n} \quad \AND \quad
    \estsupp^{\textup{BIC}} = \argmin_{S\in\suppspaceub} \frac{\|\Pi^\perp_S Y\|^2}{n} + |S|\frac{\log n}{n} \,,
\end{align*}
respectively. With the the sample size being large enough compared to the sparsity level $n\gtrsim \ubsps$, we have $n-|S|\asymp n$. Therefore, AIC and BIC are special instances of \bssa{} with different choices of $\penalt$: $\penalt=2/n$ for AIC and $\penalt=(\log n)/n$ for BIC. 
As a result, our finite sample result in Theorem~\ref{thm:opt:bss:ub:unknown} complements the asymptotic analysis of AIC, BIC type of model selection methods.
\end{remark}

It is worth highlighting the technicality involved in the proof of Theorem~\ref{thm:opt:bss:ub:unknown} (in Appendix~\ref{app:opt:bss:ub:unknown}).
When sparsity is known, Lemma~\ref{lem:opt:bss:ub:generic} (through Lemma~\ref{lem:opt:ub:bss:gap}) only needs to consider residual variances. By contrast, in the unknown sparsity case, model complexity matters: The size of $T$ may also be different than $\trusupp$.
Lemma~\ref{lem:opt:bss:gap:unknown} bounds the error probability of distinguishing $\trusupp$ and alternative support $T$ based on~\eqref{eq:opt:bss:bssa},
and illustrates the crucial role played by the additive penalty $|S|\tau$ in the estimator to differentiate $\trusupp$ from $T$ such that the error exponent shrinks fast enough.

Since the known sparsity setting is a special case of unknown sparsity, the lower bound constructions apply with minor modifications. 
We start with showing the equi-correlation $\Sigma_\mineig$ with the correlation specified by $\mineig$ satisfies the requirement for $\Sigmaspace_{\dd}^{\ubsps}(\mineig)$.
An inspection of the proof of Theorem~\ref{thm:opt:bss:lb1} reveals that, it suffices to have one single ensemble whose instances only differ in one entry of the support. Hence, we consider an ensemble of supports with size $\sps=1\le \ubsps$ and derive the bound below to match the first term in~\eqref{eq:opt:bss:ub:unknown}.
\begin{theorem}\label{thm:opt:bss:lb1:unknown}
Assuming $\mineig<1$, given $n$ i.i.d. samples from $P_{\nbeta,\Sigma,\sigma^2}$ with $(\nbeta,\Sigma,\sigma^2)\in \ubmclass$. If the sample size is bounded as
\begin{align*}
    n & \le (1-2\delta)\times \frac{\log \dd}{\betam^2\mineig / \sigma^2} \,,
\end{align*}
then for any estimator $\estsupp$ for $\trusupp=\supp(\nbeta)$,
\begin{align*}
    \inf_{\estsupp}\sup_{P\in  \ubmclass} \prob_{\nbeta,\Sigma,\sigma^2}(\estsupp\ne \trusupp) \ge \delta - \frac{\log 2}{\log \dd}
\end{align*}
\end{theorem}
The proof of Theorem~\ref{thm:opt:bss:lb1:unknown} is in Appendix~\ref{app:opt:bss:lb1:unknown}.
To match the second term in~\eqref{eq:opt:bss:ub:unknown}, since $\betaspace_{\dd,\ubsps}\subseteq \betaspace_{\dd}^{\ubsps}$, the lower bound construction in Theorem~\ref{thm:opt:bss:lb3} directly applies for unknown sparsity setting by simply changing $\sps$ to $\ubsps$.
Combined with Theorem~\ref{thm:opt:bss:lb1:unknown}, the lower bounds match the upper bound in Theorem~\ref{thm:opt:bss:ub:unknown}. Therefore, we are able to conclude \bssa{} is indeed optimal and put forth the optimality result of \bss{} from known sparsity setting to unknown sparsity.

\section{Polynomial-efficient sample complexity lower bound}\label{sec:poly}
While the general optimality of \bss{} is appealing, its computational cost prohibits its practical use. 
In the standard design where $\Sigma=I_\dd$, apart from the optimality of \bss{}, there exist other computationally efficient estimators achieving the same sample complexity, e.g. directly using the support of Lasso estimate \citep{wainwright2009sharp} or simply by marginal screening.
The natural question to ask is whether the optimal sample complexity under the general design can be achieved by more efficient estimators.
In this section, we 
give a negative answer by showing that any polynomial-efficient support estimator with known sparsity cannot avoid the restricted eigenvalue condition, establishing a gap in the sample complexity between the optimal estimator (\bss{}) and any efficient algorithms.

Our result is closely related to the established lower bound for prediction risk \citep{zhang2014lower}, thus we borrow the notion of polynomial-efficient estimator therein. We briefly introduce the most relevant concepts here; interested readers are advised to consult \citet{zhang2014lower} and books on complexity theory \citep{arora2009computational} for details.
We start with quantization for any input value $x$ to the accuracy level given by an integer $\tau$ by defining an operator $\lfloor x \rfloor_\tau := 2^{-\tau}\lfloor 2^\tau \rfloor$. Let $\text{size}(x;\tau)$ be the length of the binary representation of $\lfloor x \rfloor_\tau$, and $\text{size}(X,y;\tau)$ be the total length of the discretized data as matrix $(X,y)$.
Then the polynomial efficiency is defined by three quantities: 1) a positive integer $b$ for the number of bits needed to encode an estimator as a program; 2) a polynomial function $G$ for the discretization accuracy of the input data; 3) a polynomial function $H$ for the runtime of the program.
\begin{definition}[Polynomial-efficient support estimator]\label{defn:poly:est}
    Given polynomial functions $G: (\mathbb{Z}_+)^3\to \mathbb{R}_+$, $H: \mathbb{Z}_+\to\mathbb{R}_+$ and an integer $b\in\mathbb{Z}_+$, a support estimator $\estsupp(X,y)$ 
    is $(b,G,H)$-efficient if:
    \begin{itemize}
        \item It can be represented by a computer program encoded in $b$ bits;
        \item For every triplet $(n,\dd,\sps)$, it accepts inputs quantized to accuracy $\lfloor \cdot \rfloor_\tau$ with $\tau\in G(n,\dd,\sps)$;
        \item For every input $(X,y)$, it is guaranteed to terminate in time $H(\textup{size}(X,y;\tau))$.
    \end{itemize}
\end{definition}
We will also invoke a common conjecture in complexity theory, namely $\mathbf{NP}\not\subset\mathbf{P/poly}$, where $\mathbf{P/poly}$ consists of problems solvable in polynomial time by a Turing machine with side-input of polynomial length.

The proof is based on the idea of reduction from variable selection to prediction risk via sample splitting. Once an efficient and consistent estimator outputs the correct support on one split of the sample, the prediction risk is  $\sigma^2\sps/n$ for the OLS estimate of $\nbeta_{\trusupp}$ on the other split.
Assuming $\mathbf{NP}\not\subset \mathbf{P}/ \mathbf{poly}$, Theorem~1 in \citet{zhang2014lower} implies a lower bound for the prediction risk of a linear model for any polynomial-efficient estimator of $\nbeta$ vector.
Based on that, we show by reduction that any polynomial-efficient support estimator will have error probability lower bounded if it fails to satisfy the optimal sample complexity achieved by \bss{} multiplied by an additional factor of restricted eigenvalue of the design. 
\begin{lemma}\label{thm:poly:main}
    If $\mathbf{NP}\not\subset \mathbf{P}/ \mathbf{poly}$, then for any $\delta\in  (0,1)$, any $b\in \mathbb{Z}_+$, any polynomial functions $G:(\mathbb{Z}_+)^3 \to \mathbb{R}_+$ and $F,H: \mathbb{Z}_+ \to \mathbb{R}_+$ with $G(n,\dd,\sps) > \frac{\delta}{2\log 2}\log \sps$ for $\sps\ge 1$, there exists a sparsity level $\sps\ge 1$ such that for any $\dd\in [4\sps, F(\sps)]$, $n\in [C_1 \sps\log \dd, F(\sps)]$, and $\gamma\in [2^{-G(n,\dd,\sps)},\sps^{-\delta/2} \wedge 1/{24\sqrt{2}})$, there exists a design matrix $X\in\mathbb{R}^{n\times \dd}$ and $\nbeta\in\betaspace_{\dd,\sps}$ such that:
    \begin{enumerate}
        \item The RE constant $|\gamma(X)-\gamma|\le 2^{-G(n,\dd,\sps)}$;
        \item For any $(b,G,H)$-efficient support estimator $\estsupp$ with knowledge of $\sps$, the error probability is lower bounded as
        \begin{align*}
            \prob(\estsupp\ne\trusupp) \ge 1 \wedge \frac{C_2}{n}\times \frac{\sps^{1-\delta}\log \dd}{\max_{T\ne\trusupp}\frac{\|\Pi_T^{\perp} X_{\trusupp\setminus T}\nbeta_{\trusupp\setminus T}\|^2}{n}/\sigma^2}\times \frac{1}{\gamma^2} \,,
        \end{align*}
    \end{enumerate}
    where $C_1,C_2$ are positive constants.
\end{lemma}
The proof is in Appendix~\ref{app:poly:main}.
Actually, since $G(n,\dd,\sps)$ is a polynomial, the requirement for it to surpass $\delta \log \sps$ is easy to satisfy, e.g. $G(n,\dd,\sps)=\sps$.

Let's interpret this lower bound result.
From Lemma~\ref{thm:poly:main}, for any efficient support estimator to be consistent, i.e. $\prob(\estsupp\ne\trusupp)$ goes to zero, the sample size is required to be at least lower bounded by (since $\delta$ can be arbitrarily small)
\begin{align}\label{eq:poly:samcom}
    n \gtrsim \frac{\sps\log \dd}{\max_{T\ne\trusupp}\frac{\|\Pi_T^{\perp} X_{\trusupp\setminus T}\nbeta_{\trusupp\setminus T}\|^2}{n}/\sigma^2} \times \frac{1}{\gamma^2} \,.
\end{align}
To compare with \bss{}, let's further introduce some notations for the signals. The term $\|\Pi_T^{\perp} X_{\trusupp\setminus T}\nbeta_{\trusupp\setminus T}\|^2/n$ is the excess error in the prediction risk of OLS estimate when misspecifying the support by $T$ instead of $\trusupp$, meanwhile, it also characterizes the pairwise signal to distinguish the true support with alternative $T$. It can be viewed as a fixed design counterpart of~\eqref{eq:bss:signal}, based on whose definition, we analogously introduce the scaled version of them by the difference in supports $|\trusupp\setminus T|$:
\begin{align*}
    \Delta_u &:= \max_{T\in\suppspace\setminus\{\trusupp\}} \frac{1}{|\trusupp\setminus T|}\frac{\|\Pi_T^{\perp} X_{\trusupp\setminus T}\nbeta_{\trusupp\setminus T}\|^2}{n\sigma^2} \\
    \Delta_l &:= \min_{T\in\suppspace\setminus\{\trusupp\}} \frac{1}{|\trusupp\setminus T|}\frac{\|\Pi_T^{\perp} X_{\trusupp\setminus T}\nbeta_{\trusupp\setminus T}\|^2}{n\sigma^2} \,.
\end{align*}
Unlike model classes~\eqref{eq:mclass} for uniform bound, the fixed design signals $\Delta_u$ and $\Delta_l$ are pointwise quantities and depend on $X,\nbeta,\trusupp,\sigma^2$. 
The only difference between them is whether the maximum or minimum of the pairwise signals is taken over all possible alternatives.
These notations are helpful to draw conclusion on the same page. 
Using $\Delta_l$, we can derive below the fixed design version of the sample complexity upper bound for \bss{} as~\eqref{eq:opt:bss:ub:generic}, whose proof is exactly the same as that of Lemma~\ref{lem:opt:bss:ub:generic} thus omitted:
\begin{align}\label{eq:opt:bss:ub:fd}
    n\gtrsim \frac{\log \dd}{\Delta_l} \vee \log\binom{\dd-\sps}{\sps}\,.
\end{align}
Using $\Delta_u$, we can simplify~\eqref{eq:poly:samcom} by upper bounding the denominator by $\sps\Delta_u$ and compare with \bss{} on the first term of ~\eqref{eq:opt:bss:ub:fd} (the second term is required for any estimators by Theorem~\ref{thm:opt:bss:lb3}):
\begin{theorem}\label{thm:poly:gap}
    Under the conditions in Lemma~\ref{thm:poly:main}, the sample complexity gap between the statistical optimality and any polynomial-efficient support estimator is 
    \begin{align}\label{eq:poly:gap}
    \underbrace{\frac{\log \dd}{\Delta_l}}_{\text{optimal by }\bss{}} \quad \text{vs.}  \quad \underbrace{\frac{\log \dd}{\Delta_u} \times \frac{1}{\gamma^2}}_{\text{polynomial-efficient lower bound}}\,.
\end{align}
\end{theorem}
Theorem~\ref{thm:poly:gap} is implied by Lemma~\ref{thm:poly:main}. 
Recall that $\gamma(X)$ is the RE of the design matrix (cf.~Definition~\ref{defn:poly:RE}). Assuming $\Delta_l\asymp \Delta_u$, we observe the unavoidable gap between optimal sample complexity (achieved by \bss{}) and any polynomial-efficient estimator by a (squared) factor of restricted eigenvalue.
This gap consolidates the optimal position of \bss{}, especially in the general design setting. In other word, there is no polynomial-efficient substitute for \bss{} that can attain the same performance.

The lower bound in~\eqref{eq:poly:gap} is also suggestive in its connection with existing results for polynomial-efficient estimators based on $\ell_1$-regularization.
For example, many such methods estimate the support by applying various thresholding techniques to a Lasso-based estimator of $\nbeta$, e.g. \citet{meinshausen2009lasso} and \citet{ndaoud2020optimal}. These methods rely on the consistency of estimating $\nbeta$, and thus typically have dependence on the restricted eigenvalue  \citep{raskutti2010restricted}. However, these results also impose additional assumptions such as incoherent designs. Under such stronger assumptions, it is possible that~\eqref{eq:poly:gap} reflects the optimal sample complexity for any polynomial-efficient support estimator.

\begin{remark}
    The presumption of $\Delta_l\asymp \Delta_u$ in the comparison above requires each covariate in the true support contributes same order effect to $Y$, via the interplay of conditional variances and corresponding entries in $\nbeta_{\trusupp}$. 
    The equi-correlation covariance $\Sigma_\mineig$ for the lower bound construction in Theorem~\ref{thm:opt:bss:lb1} is one example.
    To exactly verify this is nontrivial since the existence claim for the $\nbeta$ vector in Theorem~1 of \citet{zhang2014lower} does not give an explicit construction.
    That being said, it remains reasonable to expect $\Delta_l\asymp \Delta_u$ to hold, given the symmetric nature of the construction of the hard instance for the design matrix $X$.
\end{remark}

\begin{remark}
    One recent development proposes a polynomial-time solution to \bss{} \citep{zhu2020polynomial}, however, like most efficient support estimators, this approach requires additional conditions on the covariance. Specifically, it imposes Sparse restricted condition (SRC), which is closely related to Restricted isometry property (RIP) \citep{candes2005decoding} and requires for any $|T|\le 2\sps$, and $u\ne 0$, $c_- \le \|X_T u\|^2 / (n\|u\|^2) \le c_+$
    for some constants $c_-$ and $c_+$. This effectively demands any submatrix of the covariance is close to being orthonormal, and is easily violated under general designs. 
    For instance, consider the equi-correlation covariance $\Sigma_\mineig$ used in the lower bound construction of Theorem~\ref{thm:opt:bss:lb1}. For any $T$ with $|T|=2\sps$, $u=\mathbf{1}_{2\sps}$, we have as $\sps\to \infty$, $u\T \Sigma_{w,TT} u / \|u\|^2 = (2\sps-1)(1-\mineig) + 1 \to \infty$ diverges and thus violates SRC in expectation.
\end{remark}

\section{Conclusion}
In this paper, we studied the variable selection problem in the prototypical Gaussian linear model. We validate the folklore claim that the classic combinatorial algorithm \bss{} is minimax optimal for this problem, providing the optimal sample complexity with exact dependence on the design covariance and other relevant problem parameters. 
Additionally, we have extended the optimality of \bss{} to the setting where the exact sparsity level is unknown by analyzing an information criterion-type estimator, while whether the adaptivity to the unknown sparsity can be achieved remains an important future direction.

To further affirm the optimal role of \bss{}, we provide a performance lower bound for any polynomial-efficient estimator. The lower bound reveals a gap to the optimality by a factor of restricted eigenvalue of the design matrix. The lower bound is based on a reduction to prediction risk.
This gap gives a negative answer to the question of whether the optimality can be achieved computationally efficiently. This result also draws connection to many $\ell_1$-based methods, suggesting the potential of them to achieve the polynomial-efficient lower bound, which is left for future work as well.

\appendix

\section{Proof of Lemma~\ref{lem:opt:bss:ub:generic}}\label{app:opt:bss:ub}
\begin{proof}
It is easy to see that the estimator succeeds when $\|\Pi_{\trusupp}^\perp Y\|^2$ is the smallest, i.e.
\begin{align*}
    \prob(\estsupp^{\bss{}} \ne \trusupp)  & = \prob\bigg[ \bigcup_{T\in\suppspace\setminus\{\trusupp\}}\bigg\{\|\Pi_T^\perp Y\|^2 - \|\Pi_{\trusupp}^\perp Y\|^2< 0 \bigg\}\bigg] \\
    & \le \sum_{\ell=1}^\sps\sum_{\substack{T\in\suppspace\setminus\{\trusupp\}\\ |\suppspace\setminus T|=\ell}}\prob\bigg[ \|\Pi_T^\perp Y\|^2 - \|\Pi_{\trusupp}^\perp Y\|^2< 0\bigg] \,.
\end{align*}
Now we introduce a deviation bound for the error probability, whose proof can be found below. For any $S,T\in\suppspace$, we define the signal to distinguish $S$ and $T$ to be
\begin{align*}
    \signalone(S,T) := \nbeta\T_{S\setminus T} \Sigma_{S\setminus T\given T} \nbeta_{S\setminus T}/ \sigma^2 \,.
\end{align*}
Then if $|\trusupp\setminus T| = \ell $, it is easy to see that $\signalone(\trusupp,T) \ge \ell \signalone$.
\begin{lemma}\label{lem:opt:ub:bss:gap}
If $n-\sps \ge \frac{32}{\signalone}$, then for any $T\in\suppspace\setminus \trusupp$,
\begin{align*}
    \prob\bigg[ \|\Pi_T^\perp Y\|^2 - \|\Pi_{\trusupp}^\perp Y\|^2< 0\bigg] \le 5\exp\bigg(-(n-\sps) \frac{\min(\signalone(\trusupp,T),1)}{1024}\bigg) \,.
\end{align*}
\end{lemma}
Applying Lemma~\ref{lem:opt:ub:bss:gap}, we have
\begin{align*}
    \prob(\estsupp^{\bss{}} \ne \trusupp)  & \le \sum_{\ell=1}^\sps\sum_{\substack{T\in\suppspace\setminus\{\trusupp\} \\ |\trusupp\setminus T|=\ell}}\prob\bigg[ \|\Pi_T^\perp Y\|^2 - \|\Pi_{\trusupp}^\perp Y\|^2< 0\bigg] \\
    & \le  \sum_{\ell=1}^\sps\sum_{\substack{T\in\suppspace\setminus\{\trusupp\} \\ |\trusupp\setminus T|=\ell}} 5\exp\bigg(-(n-\sps) \frac{\min(\signalone(\trusupp,T),1)}{1024}\bigg)\\
    & \le \max_\ell \sps\times \binom{\sps}{\ell}\binom{\dd-\sps}{\ell}\times 5 \exp\bigg(-(n-\sps) \frac{\min(\ell \signalone,1)}{1024}\bigg) \,.
\end{align*}
Since $\sps\le \dd/2$, $\binom{\sps}{\ell} \le \binom{\dd-\sps}{\ell}$ and 
\begin{align*}
    \log 5\sps \le \max_\ell \log 5\binom{\sps}{\ell} \le \max_\ell 2\log \binom{\sps}{\ell}\,.
\end{align*}
Therefore, the error probability
\begin{align*}
    \prob(\estsupp^{\bss{}} \ne \trusupp)  & \le\max_\ell \exp\bigg(\log 5\sps + \log \binom{\sps}{\ell} + \log \binom{\dd-\sps}{\ell} -(n-\sps) \frac{\min(\ell \signalone,1)}{1024}\bigg)\\
    & \le\max_\ell \exp\bigg(4\log \binom{\dd-\sps}{\ell} -(n-\sps) \frac{\min(\ell \signalone,1)}{1024}\bigg)\,.
\end{align*}
Setting the RHS to be smaller than $\delta$ for all $\ell\in[\sps]$, we have desired sample complexity.
\end{proof}
\begin{proof}[Proof of Lemma~\ref{lem:opt:ub:bss:gap}]
\begin{align*}
    \prob\bigg( \|\Pi_T^\perp Y\|^2 - \|\Pi_{\trusupp}^\perp Y\|^2< 0 \bigg) & = \prob\bigg(\frac{\|\Pi_T^\perp Y\|^2 - \|\Pi_{\trusupp}^\perp Y\|^2}{(n-\sps)\sigma^2} < 0\bigg) \\
    & \le \prob\bigg(\frac{\|\Pi_T^\perp Y\|^2 - \|\Pi_{\trusupp}^\perp Y\|^2}{(n-\sps)\sigma^2} \le \signalone(\trusupp,T) /4 \bigg) \\
    & \le \prob\bigg(\frac{\big| \|\Pi_T^\perp Y\|^2 - \|\Pi_{T}^\perp \neps \|^2 \big|}{(n-\sps)\sigma^2} \le \signalone(\trusupp,T)/2\bigg)\\
    & \qquad + \prob\bigg(\frac{\big| \|\Pi_T^\perp \neps\|^2 - \|\Pi_{\trusupp}^\perp \neps\|^2 \big|}{(n-\sps)\sigma^2} \ge \signalone(\trusupp,T)/4\bigg)  \,.
\end{align*}
We bound these two terms separately using the lemmas below. 
\begin{lemma}\label{lem:opt:ub:bss:gap:1}
If $n-\sps\ge \frac{32}{\signalone}$, then
\begin{align*}
    \prob\bigg(\frac{\big| \|\Pi_T^\perp \neps\|^2 - \|\Pi_{\trusupp}^\perp \neps\|^2 \big|}{(n-\sps)\sigma^2} \ge \signalone(\trusupp,T) / 4\bigg)  \le 2\exp\bigg( - (n-\sps) \frac{\signalone(\trusupp,T)}{32}\bigg) \,.
\end{align*}
\end{lemma}
\begin{lemma}\label{lem:opt:ub:bss:gap:2}
\begin{align*}
    \prob\bigg(\frac{\big| \|\Pi_T^\perp Y\|^2 - \|\Pi_{T}^\perp \neps \|^2 \big|}{(n-\sps)\sigma^2} \le \signalone(\trusupp,T)/2\bigg) & \le  2\exp\bigg(-(n-\sps)\frac{\min(\signalone(\trusupp,T),\sqrt{\signalone(\trusupp,T)})}{1024}\bigg)\\
    &+ \exp\bigg(-(n-\sps) / 256\bigg) \,.
\end{align*}
\end{lemma}
Combining Lemma~\ref{lem:opt:ub:bss:gap:1} and \ref{lem:opt:ub:bss:gap:2}, we complete the proof.
\end{proof}
\begin{proof}[Proof of Lemma~\ref{lem:opt:ub:bss:gap:1}]
\begin{align*}
    \|\Pi_T^\perp \neps\|^2 /\sigma^2- \|\Pi_{\trusupp}^\perp\neps\|^2/\sigma^2 &= \neps\T \bigg( (I_n - X_T(X_T\T X_T)^{-1} X_T) - (I_n - X_{\trusupp}(X_{\trusupp}\T X_{\trusupp})^{-1} X_{\trusupp})\bigg)\neps/\sigma^2 \\
    & = \neps\T \bigg(X_{\trusupp}(X_{\trusupp}\T X_{\trusupp})^{-1} X_{\trusupp} - X_T(X_T\T X_T)^{-1} X_T) \bigg)\neps /\sigma^2 \\
    & = \|(\Pi_{\trusupp}-\Pi_{\trusupp\cap T})\neps\|^2/\sigma^2 - \|(\Pi_{T} - \Pi_{\trusupp\cap T})\neps\|^2 /\sigma^2 \\
    & = Z - \widetilde{Z}\,,
\end{align*}
where $Z,\widetilde{Z}\sim \chi^2_\ell$ given $X_T$ and $X_{\trusupp}$ because both $\Pi_{\trusupp}-\Pi_{\trusupp\cap T} \AND \Pi_{T}-\Pi_{\trusupp\cap T}$ are projection matrix with trace equal to $|\trusupp\setminus T|=|T\setminus \trusupp|=\ell$. Therefore, 
\begin{align*}
    \prob\bigg(\frac{\big| \|\Pi_T^\perp \neps\|^2 - \|\Pi_{\trusupp}^\perp \neps\|^2 \big|}{(n-\sps)\sigma^2} \ge \signalone(\trusupp,T) / 4\bigg) & = \prob\bigg(\frac{\big| Z - \widetilde{Z} \big|}{(n-\sps)} \ge \signalone(\trusupp,T) / 4\bigg) \\
    & \le \prob\bigg(\frac{|Z-\ell|}{\ell}\ge \frac{(n-\sps)\signalone(\trusupp,T)}{8\ell}\bigg) \\
    & \qquad + \prob\bigg(\frac{|\widetilde{Z}-\ell|}{\ell}\ge \frac{(n-\sps)\signalone(\trusupp,T)}{8\ell}\bigg)\\
    & = 2\prob\bigg(\frac{|Z-\ell|}{\ell}\ge \frac{(n-\sps)\signalone(\trusupp,T)}{8\ell}\bigg) \\
    & \le 2\exp\bigg(-(n-\sps)\frac{\signalone(\trusupp,T)}{32}\bigg)\,.
\end{align*}
We apply Lemma~\ref{lem:chisq} for the last inequality since $n-\sps\ge \frac{32}{\signalone} =\frac{32\ell}{\ell\signalone} \ge \frac{32}{\signalone(\trusupp,T)}$.
\end{proof}
\begin{proof}[Proof of Lemma~\ref{lem:opt:ub:bss:gap:2}]
We can decompose the Gaussian variables
\begin{align*}
    X_{\trusupp\setminus T} = X_T(\Sigma_{TT})^{-1}\Sigma_{T(\trusupp\setminus T)} + E_{\trusupp\setminus T}\,,
\end{align*}
where each row of $E_{\trusupp\setminus T}$ is i.i.d. from $\mathcal{N}(0,\Sigma_{\trusupp\setminus T\given T})$ and $E_{\trusupp\setminus T} \indep X_T$. Thus
\begin{align*}
    \Pi_T^\perp Y & = \Pi_T^\perp (X_{\trusupp}\nbeta_{\trusupp} + \neps) \\
    & = \Pi_T^\perp(X_{\trusupp\cap T} \nbeta_{\trusupp\cap T} + X_{\trusupp\setminus T}\nbeta_{\trusupp\setminus T} + \neps)\\
    & = \Pi_T^\perp (X_{\trusupp\setminus T}\nbeta_{\trusupp\setminus T} + \neps) \\
    & = \Pi_T^\perp \bigg((X_T(\Sigma_{TT})^{-1}\Sigma_{T(\trusupp\setminus T)} + E_{\trusupp\setminus T})\nbeta_{\trusupp\setminus T} + \neps\bigg) \\
    & = \Pi_T^\perp (E_{\trusupp\setminus T}\nbeta_{\trusupp\setminus T} + \neps)\,,
\end{align*}
where $E_{\trusupp\setminus T}\nbeta_{\trusupp\setminus T}$ is a vector and each entry is i.i.d. from $\mathcal{N}(0,\signalone(\trusupp,T)\sigma^2)$. Therefore,
\begin{align*}
    \|\Pi_T^\perp Y\|^2 - \|\Pi_T^\perp \neps\|^2 & = \|\Pi_T^\perp E_{\trusupp\setminus T}\nbeta_{\trusupp\setminus T}\|^2 + 2\langle \Pi_T^\perp E_{\trusupp\setminus T}\nbeta_{\trusupp\setminus T}, \neps \rangle \\
    &  = \sigma^2\signalone(\trusupp,T)\|\Pi_T^\perp U\|^2 + 2\sigma^2 \sqrt{\signalone(\trusupp,T)} \langle \Pi_T^\perp U, \Pi_T^\perp U_\neps\rangle\,,
\end{align*}
where $ U,U_\neps\sim\mathcal{N}(0,I_n)$ and $U\indep U_\neps$. Then
\begin{align*}
    \prob\bigg(\frac{\big| \|\Pi_T^\perp Y\|^2 - \|\Pi_{T}^\perp \neps \|^2 \big|}{(n-\sps)\sigma^2} \le \signalone(\trusupp,T)/2\bigg) & \le \prob\bigg(\frac{\sigma^2\signalone(\trusupp,T)\|\Pi_T^\perp U\|^2}{\sigma^2(n-\sps)}\le \frac{3}{4}\signalone(\trusupp,T)\bigg) \\
    & + \prob\bigg(\frac{2\sigma^2\sqrt{\signalone(\trusupp,T)}\langle \Pi_T^\perp U,\Pi_T^\perp U_\neps\rangle}{\sigma^2(n-\sps)} \le -\frac{1}{4}\signalone(\trusupp,T)\bigg)\,.
\end{align*}
For the first term, apply Lemma~\ref{lem:chisq},
\begin{align*}
    \prob\bigg(\frac{\sigma^2\signalone(\trusupp,T)\|\Pi_T^\perp U\|^2}{\sigma^2(n-\sps)}\le \frac{3}{4}\signalone(\trusupp,T)\bigg)  & = \prob\bigg(\frac{\chi^2_{n-\sps}}{n-\sps} \le \frac{3}{4}\bigg)\\
    & = \prob\bigg(\frac{\chi^2_{n-\sps}}{n-\sps}-1 \le -\frac{1}{4}\bigg)\\
    & \le \exp(-(n-\sps)/256) \,.
\end{align*}
For the second term, since
\begin{align*}
    2\langle \Pi_T^\perp U,\Pi_T^\perp U_\neps\rangle & = \frac{1}{2}\bigg( \|\Pi_T^\perp(U+U_\neps)\|^2 - \|\Pi_T^\perp(U-U_\neps)\|^2\bigg) \\
    & = \|\Pi_T^\perp\frac{U+U_\neps}{\sqrt{2}}\|^2 - \|\Pi_T^\perp\frac{U-U_\neps}{\sqrt{2}}\|^2 \\
    & = W-\widetilde{W} \,,
\end{align*}
where $W,\widetilde{W}\sim \chi^2_{n-\sps}$. Therefore, apply Lemma~\ref{lem:chisq},
\begin{align*}
    \prob\bigg(\frac{2\sigma^2\sqrt{\signalone(\trusupp,T)}\langle \Pi_T^\perp U, \Pi_T^\perp U_\neps\rangle}{\sigma^2(n-\sps)} &\le -\frac{1}{4}\signalone(\trusupp,T)\bigg) \\
    & = \prob\bigg(\frac{W-\widetilde{W}}{n-\sps} \le -\frac{\sqrt{\signalone(\trusupp,T)}}{4}\bigg) \\
    & \le \prob\bigg(\frac{|W-(n-\sps)|}{n-\sps} \ge \frac{\sqrt{\signalone(\trusupp,T)}}{8}\bigg) \\
    & \qquad +\prob\bigg(\frac{|\widetilde{W}-(n-\sps)|}{n-\sps} \ge \frac{\sqrt{\signalone(\trusupp,T)}}{8}\bigg) \\
    & \le 2\exp\bigg(-(n-\sps)\frac{\min(\signalone(\trusupp,T),\sqrt{\signalone(\trusupp,T)})}{1024}\bigg)\,.
\end{align*}
\end{proof}

\section{Proof of Theorem~\ref{thm:opt:bss:lb1}}\label{app:opt:bss:lb1}
\begin{proof}
Consider a covariance matrix of $X$:
\begin{align*}
    \Sigma = (1-\rho) I_\dd + \rho \mathbf{1}_\dd\mathbf{1}_\dd\T
\end{align*}
where $\rho = 1 - \mineig$. Then for any distinct $S,T\in\suppspace$ with $|S\setminus T| = r$, we can calculate the conditional covariance matrix
\begin{align*}
    \Sigma_{S\setminus T\given T} & = \Sigma_{(S\setminus T)(S\setminus T)} - \Sigma_{(S\setminus T)T}\Sigma^{-1}_{TT}\Sigma_{T(S\setminus T)} \\
    & = (1-\rho)I_r + \rho \mathbf{1}_r\mathbf{1}_r\T - \rho \mathbf{1}_r\mathbf{1}_\sps\T \times \big((1-\rho)I_\sps + \rho \mathbf{1}_\sps\mathbf{1}_\sps\T\big)^{-1} \times \rho \mathbf{1}_\sps\mathbf{1}_r\T \\
    & = (1-\rho)I_r + \rho \mathbf{1}_r\mathbf{1}_r\T - \rho \mathbf{1}_r\mathbf{1}_\sps\T \times \bigg(\frac{1}{1-\rho}\big(I_\sps - \frac{\rho}{1 - \rho + \sps\rho}\mathbf{1}_\sps\mathbf{1}_\sps\T\big)\bigg) \times \rho \mathbf{1}_\sps\mathbf{1}_r\T \\
    & = (1-\rho)\bigg(I_r + \frac{\rho}{1-\rho + \sps\rho} \mathbf{1}_r\mathbf{1}_r\T\bigg)\,,
\end{align*}
then the minimum eigenvalue is
\begin{align*}
    \lambda_{\min}(\Sigma_{S\setminus T\given T}) = \begin{cases}
    (1 - \rho) \times (1+\frac{\rho}{1-\rho+\sps\rho}) & r = 1 \\
    1-\rho & r\ge 2
    \end{cases} \,.
\end{align*}
Since $\Sigma_{S\setminus T\given T}$ is independent with the choice of $(S,T)$ but only depends on $|S\setminus T|$, this covariance matrix $\Sigma$ satisfies the requirement on $\Sigmaspace_{\dd,\sps}(\mineig)$:
\begin{align*}
    \min_{S\in\suppspace}\min_{T\in\suppspace\setminus S} \lambda_{\min}(\Sigma_{S\setminus T\given T}) = 1 - \rho = \mineig \,.
\end{align*}
Now we can construct $\sps$ many ensembles to establish $\sps$ lower bounds, while they will lead to only one in the end. For each of them, we fix the covariance matrix $\Sigma$ and coefficient vector $\nbeta = \betam\mathbf{1}_\dd$, construct the ensemble solely by varying support. For the $\ell$-th ensemble ($\ell=1,2,\cdots,\sps$), let
\begin{align*}
    \mathcal{S}'_\ell :=  \bigg\{S'\subseteq \{\sps,\sps+1,\cdots,\dd\}: |S'|=\ell \bigg\}\,.
\end{align*}
Thus $|\mathcal{S}'_\ell| = \binom{\dd-\sps}{\ell}$. We consider set of supports:
\begin{align*}
    \mathcal{S}_\ell :=  \bigg\{S : S = \{1,2,\cdots,\sps-\ell\}\cup S', S'\in\mathcal{S}'_\ell \bigg\}\,.
\end{align*}
Thus $|\mathcal{S}_\ell|=|\mathcal{S}'_\ell| = \binom{\dd-\sps}{\ell}$ and each element $S\in \mathcal{S}_\ell$ determines a model $Y = X_S\T \nbeta_S +\neps$. For any two supports $S,T\in\mathcal{S}_\ell$, write
\begin{align*}
    S & = \{1,2,\cdots,\sps-\ell\}\cup S' \\
    T & = \{1,2,\cdots,\sps-\ell\}\cup T' \,,
\end{align*}
with $S',T'\in\mathcal{S}'_\ell$. Now we try to calculate the KL divergence between two models specified by $S$ and $T$. We further denote $S'' = S'\setminus T'$, and $T''=T'\setminus S'$ with $|S''|=|T''|=r\le \ell$, and the models determined by $S$ and $T$ to be $P_S$ and $P_T$. Therefore,
\begin{align*}
    \mathbf{KL}(P_S\| P_T) & = \E_{P_S} \log \frac{P_S}{P_T} \\
    & = \E_{P_S} \log \frac{\exp\bigg(-(Y - X_S\T \nbeta_S)^2 / 2\sigma^2\bigg)}{\exp\bigg(-(Y - X_T\T \nbeta_T)^2/ 2\sigma^2\bigg)} \\
    & = \E_X\E_\neps \frac{1}{2\sigma^2}\bigg( (X_S\T\nbeta_S - X_T\T\nbeta_T + \neps)^2 - \neps^2 \bigg) \\
    & = \E_X  (X_S\T\nbeta_S - X_T\T\nbeta_T)^2 / 2\sigma^2\\ 
    & = \E_X (X_{S''}\T\nbeta_{S''} - X_{T''}\T\nbeta_{T''})^2 / 2\sigma^2 \\ 
    & = \mathbf{1}_r\T (\Sigma_s + \Sigma_t - \Sigma_{st} - \Sigma_{ts}) \mathbf{1}_r\times \frac{\betam^2}{2\sigma^2}\,,
\end{align*}
where $\Sigma_s:= \Sigma_{S''S''}, \Sigma_t := \Sigma_{T''T''}, \Sigma_{st} := \Sigma_{S''T''},\Sigma_{ts}:=\Sigma_{T''S''}$. Then
\begin{align*}
     \mathbf{1}_r\T(\Sigma_s + \Sigma_t - \Sigma_{st} - \Sigma_{ts}) \mathbf{1}_r & =  \mathbf{1}_r\T(\Sigma_s - \Sigma_{st}\Sigma_t^{-1}\Sigma_{ts}) \mathbf{1}_r +  \mathbf{1}_r\T(\Sigma_t - \Sigma_{ts}\Sigma_s^{-1}\Sigma_{st})  \mathbf{1}_r\\
    & \qquad +  \mathbf{1}_r\T ( \Sigma_{st}\Sigma_t^{-1}(\Sigma_{ts} - \Sigma_t )) \mathbf{1}_r + \mathbf{1}_r\T ( \Sigma_{ts}\Sigma_s^{-1}(\Sigma_{st} - \Sigma_s )) \mathbf{1}_r \,.
\end{align*}
The first two terms are the same, which are
\begin{align*}
    \mathbf{1}_r\T(\Sigma_s - \Sigma_{st}\Sigma_t^{-1}\Sigma_{ts}) \mathbf{1}_r &  = \mathbf{1}_r\T\bigg[ (1-\rho)\big(I_r + \frac{\rho}{1-\rho + r\rho}\mathbf{1}_r\mathbf{1}_r\T\big) \bigg]\mathbf{1}_r\\
    & = r(1-\rho)\times \bigg(1 + \frac{r\rho}{1-\rho+r\rho}\bigg)\\
    & \le 2r(1-\rho)\\
    & = 2r\mineig \le 2\ell \mineig \,.
\end{align*}
The last two terms are the same, which are
\begin{align*}
    \mathbf{1}_r\T ( \Sigma_{st}\Sigma_t^{-1}(\Sigma_{ts} - \Sigma_s )) \mathbf{1}_r & = \mathbf{1}_r\T \bigg[\rho\mathbf{1}_r\mathbf{1}_r\T \times \bigg(\frac{1}{1-\rho}\big(I_r - \frac{\rho}{1-\rho+r\rho}\mathbf{1}_r\mathbf{1}_r\T\big)\bigg)\times (\rho-1)I_r \bigg] \mathbf{1}_r \\
    & = \frac{-\rho(1-\rho)}{1-\rho+r\rho}\times r^2 \le 0 \,.
\end{align*}
Thus $\mathbf{KL}(P_S\|P_T) \le 2\ell\betam^2\mineig/\sigma^2$, which holds for any two $S,T\in \mathcal{S}_\ell$ and leads to a upper bound for any two models in the $\ell$-th ensemble.
Finally, for the $\ell$-th ensemble, we apply Fano's inequality Corollary~\ref{coro:fano} with KL divergence upper bound $2\ell\betam^2\mineig/\sigma^2$ and ensemble cardinality $\binom{\dd-\sps}{\ell}$, which completes the proof.
\end{proof}

\section{Proof of Theorem~\ref{thm:opt:bss:ub:unknown}}\label{app:opt:bss:ub:unknown}
\begin{proof}
    Recall that $|\trusupp|=\sps\le \ubsps$.
    Denote the event that $\trusupp$ beats an opponent $T$ with $|T|=j$:
    \begin{align*}
        \mathcal{E}(T,j) = \bigg\{\frac{\|\Pi^\perp_{\trusupp} Y\|^2}{n-\sps} + \frac{\sps}{4}\betam^2\mineig \le \frac{\|\Pi^\perp_T Y\|^2}{n-j} + \frac{j}{4}\betam^2\mineig\bigg\}\,,
    \end{align*}
    then the estimator succeeds with 
    \begin{align*}
        \prob(\estsupp^{\bssa{}}=\trusupp) = \prob\bigg( \bigcap_{j\in\{1,2,\ldots,\ubsps\}}\bigcap_{T\in\supps_{\dd,j}\setminus \{\trusupp\}} \mathcal{E}(T,j) \bigg) \,.
    \end{align*}
    Therefore, use $\ell$ for the distance from $j$ to $\sps$, i.e. $\ell=|j-\sps|$,
    \begin{align*}
        \prob(\estsupp^{\bssa{}}\ne \trusupp) & = \prob\bigg( \bigcup_{j\in[\ubsps]}\bigcup_{T\in\supps_{\dd,j}\setminus \{\trusupp\}} \overline{\mathcal{E}(T,j)} \bigg) \\
        & \le \sum_{T\in\suppspace\setminus\{\trusupp\}}\prob(\overline{\mathcal{E}(T,\sps)}) +  \sum_{j\ne \sps}\sum_{T\in\supps_{\dd,j}}\prob(\overline{\mathcal{E}(T,j)})\\
        & = \sum_{T\in\suppspace\setminus\{\trusupp\}}\prob(\overline{\mathcal{E}(T,\sps)}) \\
        & \qquad + \sum_{\ell=1}^{\sps}\sum_{T\in\supps_{\dd,\sps-\ell}}\prob(\overline{\mathcal{E}(T,\sps-\ell)})  + \sum_{\ell=1}^{\ubsps-\sps}\sum_{T\in\supps_{\dd,\sps+\ell}}\prob(\overline{\mathcal{E}(T,\sps+\ell)}) \,.
    \end{align*}
    The first term is controlled by Theorem~\ref{thm:opt:bss:ub}, now let's look at remaining two. Use $k:=|\trusupp\cap T|$ for the overlap between $T$ and $\trusupp$, define 
    \begin{align*}
        A_1 & := \sum_{\ell=1}^{\sps}\sum_{T\in\supps_{\dd,\sps-\ell}}\prob(\overline{\mathcal{E}(T,\sps-\ell)})  =\sum_{\ell=1}^{\sps} \sum_{k=0}^{\sps-\ell} \sum_{\overset{T\in\supps_{\dd,\sps-\ell}}{|T\cap \trusupp| = k} }\prob(\overline{\mathcal{E}(T,\sps-\ell)}) \\
        A_2 & :=\sum_{\ell=1}^{\ubsps-\sps}\sum_{T\in\supps_{\dd,\sps+\ell}}\prob(\overline{\mathcal{E}(T,\sps+\ell)}) 
        =  \sum_{\ell=1}^{\ubsps-\sps}\sum_{k=0}^{\sps}\sum_{\overset{T\in\supps_{\dd,\sps+\ell}}{|T\cap \trusupp| = k}}\prob(\overline{\mathcal{E}(T,\sps+\ell)}) \,.
    \end{align*}
    The cardinality of the innermost sums of $A_1 \AND A_2$ are
    \begin{align*}
        & \binom{\sps}{k}\binom{\dd-\sps}{\sps-k - \ell} = \binom{\sps}{\sps-k}\binom{\dd-\sps}{\sps-k - \ell} \le \binom{\dd-\sps}{\sps-k}^2 \\
        &\binom{\sps}{k}\binom{\dd-\sps}{\sps-k + \ell} = \binom{\sps}{\sps-k}\binom{\dd-\sps}{\sps-k + \ell} \le \binom{\dd-\sps}{\sps-k + \ell}^2
    \end{align*}
    respectively. The last inequality is because $\sps\le\ubsps\le \dd/2$, $\sps-k+\ell\le \ubsps$, then $\binom{\sps}{a}\le \binom{\dd-\sps}{a}$ for $a\le \sps$. Now we analyze the error probability respectively. Denote $\upsilon := \mineig\betam^2 / \sigma^2$ as a short hand.
    
    For $|T|=\sps-\ell$ and $|T\cap \trusupp|=k$, we have $|\trusupp\setminus T|=\sps-k \ge \ell$, $|T\setminus \trusupp| =\sps-\ell - k$.
    \begin{align*}
        \prob(\overline{\mathcal{E}(T,\sps-\ell)}) & = \prob\bigg( \frac{\|\Pi^\perp_{T}Y\|^2/\sigma^2}{n-(\sps-\ell)} - \frac{\|\Pi^\perp_{\trusupp}Y\|^2/\sigma^2}{n-\sps} \le \frac{\ell}{4}\mineig\betam^2/\sigma^2\bigg)   \\
        & \le \prob\bigg( \frac{\|\Pi^\perp_{T}Y\|^2/\sigma^2}{n-(\sps-\ell)} - \frac{\|\Pi^\perp_{\trusupp}Y\|^2/\sigma^2}{n-\sps} \le \frac{|\trusupp\setminus T|}{4}\upsilon\bigg) \\
        & \le \prob\bigg( \frac{\|\Pi^\perp_{T}Y\|^2/\sigma^2}{n-(\sps-\ell)} - \frac{\|\Pi^\perp_{\trusupp}Y\|^2/\sigma^2}{n-\sps} \le \frac{1}{4}\signalone(\trusupp,T)\bigg) \,.
    \end{align*}
    For $|T|=\sps+\ell$ and $|T\cap \trusupp|=k$, we have $|\trusupp\setminus T|=\sps-k$, $|T\setminus \trusupp| =\sps+\ell - k$.
    \begin{align*}
        \prob(\overline{\mathcal{E}(T,\sps+\ell)}) & = \prob\bigg( \frac{\|\Pi^\perp_{T}Y\|^2/\sigma^2}{n-(\sps+\ell)} - \frac{\|\Pi^\perp_{\trusupp}Y\|^2/\sigma^2}{n-\sps} \le -\frac{\ell}{4}\mineig\betam^2/\sigma^2\bigg)   \\
        & \le \prob\bigg( \frac{\|\Pi^\perp_{T}Y\|^2/\sigma^2}{n-(\sps+\ell)} - \frac{\|\Pi^\perp_{\trusupp}Y\|^2/\sigma^2}{n-\sps} \le \frac{1}{4}\Big(\signalone(\trusupp,T) - \ell \upsilon\Big)\bigg) \,.
    \end{align*}
    Now we introduce following lemma to control the error probability. The proof will be given in the sequel.
\begin{lemma}\label{lem:opt:bss:gap:unknown}
    If $n-\ubsps \ge \frac{96}{\betam^2\mineig / \sigma^2}$, then for any $T\in\suppspaceub\setminus \{\trusupp\}$, let $\ell':=\max\{|T|-|\trusupp|,0\}$,
    \begin{align*}
        &\prob\bigg[ \frac{\|\Pi^\perp_{T}Y\|^2}{(n-|T|)\sigma^2} - \frac{\|\Pi^\perp_{\trusupp}Y\|^2}{(n-\sps)\sigma^2} \le \frac{1}{4}\Big(\signalone(\trusupp,T) - \ell' \mineig\betam^2/\sigma^2 \Big)\bigg] \\
        \le & 5\exp\bigg(-(n-\ubsps)\frac{\min\Big((|\trusupp\setminus T|+\ell')\mineig\betam^2/\sigma^2,1\Big)}{9216} + \frac{|T\setminus \trusupp|}{4}\bigg) \,.
    \end{align*}
    \end{lemma}
For each error probability in $A_1$, $\sps<j$ thus $\ell'=0$. While for each error probability in $A_2$, $\ell'=\sps-j=\ell$ , we apply Lemma~\ref{lem:opt:ub:bss:gap} correspondingly. For $A_1$, let $t:=\sps-k \in [\sps]$,
    \begin{align*}
        A_1 & = \sum_{\ell=1}^{\sps} \sum_{k=0}^{\sps-\ell} \sum_{\overset{T\in\supps_{\dd,\sps-\ell}}{|T\cap \trusupp| = k} }\prob(\overline{\mathcal{E}(T,\sps-\ell)}) \\
        & \le \sps(\sps-\ell+1)\max_{\overset{1\le \ell\le \sps}{0\le k \le \sps-\ell}} \binom{\dd-\sps}{\sps-k}^2 \max_{\overset{T\in\supps_{\dd,\sps-\ell}}{|T\cap \trusupp| = k}} \prob(\overline{\mathcal{E}(T,\sps-\ell)}) \\
        & \le \sps \ubsps\max_{\overset{1\le \ell\le \sps}{0\le k \le \sps-\ell}} 5\exp\bigg(-(n-\ubsps)\frac{\min\Big((\sps-k)\mineig\betam^2/\sigma^2,1\Big)}{9216} + \frac{\sps-k-\ell}{4} + 2\log \binom{\dd-\sps}{\sps-k}\bigg) \\
        & \le \sps \ubsps\max_{\overset{1\le \ell\le \sps}{0\le k \le \sps-\ell}} 5\exp\bigg(-(n-\ubsps)\frac{\min\Big((\sps-k)\mineig\betam^2/\sigma^2,1\Big)}{9216} + 3\log \binom{\dd-\sps}{\sps-k}\bigg) \\
        & = \sps \ubsps\max_{t\in[\sps]} 5\exp\bigg(-(n-\ubsps)\frac{\min\Big(t\mineig\betam^2/\sigma^2,1\Big)}{9216} + 3\log \binom{\dd-\sps}{t}\bigg) \,.
    \end{align*}
    For $A_2$, which is positive only when $\sps<\ubsps$, let $t := \sps-k+\ell \in [\ubsps]$,
    \begin{align*}
        A_2 &= \sum_{\ell=1}^{\ubsps-\sps}\sum_{k=0}^{\sps}\sum_{\overset{T\in\supps_{\dd,\sps+\ell}}{|T\cap \trusupp| = k}}\prob(\overline{\mathcal{E}(T,\sps+\ell)}) \\
        & \le (\ubsps-\sps)(\sps+1)\max_{\overset{1\le \ell\le \ubsps-\sps}{0\le k \le \sps}} \binom{\dd-\sps}{\sps-k+\ell}^2 \max_{\overset{T\in\supps_{\dd,\sps+\ell}}{|T\cap \trusupp| = k}} \prob(\overline{\mathcal{E}(T,\sps+\ell)}) \\
        & \le (\ubsps-\sps)\ubsps\max_{\overset{1\le \ell\le \ubsps-\sps}{0\le k \le \sps}} 5\exp\bigg(-(n-\ubsps)\frac{\min\Big((\sps-k+\ell)\mineig\betam^2/\sigma^2,1\Big)}{9216} + \frac{\sps-k+\ell}{4} + 2\log \binom{\dd-\sps}{\sps-k+\ell}\bigg) \\
        &\le (\ubsps-\sps)\ubsps\max_{\overset{1\le \ell\le \ubsps-\sps}{0\le k \le \sps}} 5\exp\bigg(-(n-\ubsps)\frac{\min\Big((\sps-k+\ell)\mineig\betam^2/\sigma^2,1\Big)}{9216} + 3\log \binom{\dd-\sps}{\sps-k+\ell}\bigg) \\
        &= (\ubsps-\sps)\ubsps\max_{t\in [\ubsps]} 5\exp\bigg(-(n-\ubsps)\frac{\min\Big(t\mineig\betam^2/\sigma^2,1\Big)}{9216} + 3\log \binom{\dd-\sps}{t}\bigg) \,.
    \end{align*}
    Therefore,
    \begin{align*}
        A_1+A_2 & \le 5\ubsps^2 \max_{t\in [\ubsps]} \exp\bigg(-(n-\ubsps)\frac{\min\Big(t\mineig\betam^2/\sigma^2,1\Big)}{9216} + 3\log \binom{\dd-\sps}{t}\bigg) \\
        & = \max_{t\in [\ubsps]} \exp\bigg(-(n-\ubsps)\frac{\min\Big(t\mineig\betam^2/\sigma^2,1\Big)}{9216} + 3\log \binom{\dd-\sps}{t} + \log(5\ubsps^2)\bigg) \,.
    \end{align*}
    Since
    \begin{align*}
        \log (5\ubsps^2) &= \log 5 + 2\log \ubsps \\
        & \le \log 5 + 2\max_{t\in[\ubsps]}\log \binom{\ubsps}{t} \\
        & \le 3\max_{t\in[\ubsps]}\log \binom{\ubsps}{t} \\
        & \le 3\max_{t\in[\ubsps]}\log \binom{\dd-\sps}{t} \,,
    \end{align*}
    we have
    \begin{align*}
        A_1+A_2 & \le \max_{t\in [\ubsps]} \exp\bigg(-(n-\ubsps)\frac{\min\Big(t\mineig\betam^2/\sigma^2,1\Big)}{9216} + 6 \log \binom{\dd-\sps}{t}\bigg) \,.
    \end{align*}
    Combined with Theorem~\ref{thm:opt:bss:ub}, we have following error probability,
    \begin{align*}
        \prob(\estsupp^{\bssa{}}\ne S) & \le 2\max_{t\in [\ubsps]} \exp\bigg(-(n-\ubsps)\frac{\min\Big(t\mineig\betam^2/\sigma^2,1\Big)}{9216} + 6 \log \binom{\dd-\sps}{t}\bigg) \\
        & \le 2\max_{t\in [\ubsps]} \exp\bigg(-(n-\ubsps)\frac{\min\Big(t\mineig\betam^2/\sigma^2,1\Big)}{9216} + 6 \log \binom{\dd}{t}\bigg)\,.
    \end{align*}
    Setting the RHS to be smaller than $\delta$ leads to desired sample complexity.   
\end{proof}
\begin{proof}[Proof of Lemma~\ref{lem:opt:bss:gap:unknown}]
Let $|T|=j\le \ubsps$, then
\begin{align*}
    \frac{\|\Pi^\perp_{T}Y\|^2/\sigma^2}{n-j} - \frac{\|\Pi^\perp_{\trusupp}Y\|^2/\sigma^2}{n-\sps} & = \frac{\|\Pi^\perp_T Y\|^2/\sigma^2 - \|\Pi^\perp_T \neps\|^2/\sigma^2}{n-j}\\
    & \qquad + \frac{\|\Pi^\perp_T \neps\|^2/\sigma^2}{n-j} - \frac{\|\Pi^\perp_{\trusupp}\neps\|^2/\sigma^2}{n-\sps}
\end{align*}
Similar to the proof of Lemma~\ref{lem:opt:ub:bss:gap:2}, we can write 
\begin{align*}
    \Pi_T^\perp Y = \Pi_T^\perp (E_{\trusupp\setminus T}\nbeta_{\trusupp\setminus T} + \neps)\,,
\end{align*}
where $E_{\trusupp\setminus T}\nbeta_{\trusupp\setminus T}$ is a random vector vector independent with $T$ and each entry is i.i.d. from $\mathcal{N}(0,\signalone(\trusupp,T)\sigma^2)$. Therefore,
\begin{align*}
    \frac{\|\Pi^\perp_T Y\|^2/\sigma^2 - \|\Pi^\perp_T \neps\|^2/\sigma^2}{n-j} & = \frac{\|\Pi_T^\perp E_{\trusupp\setminus T}\nbeta_{\trusupp\setminus T}\|^2 + 2\langle \Pi_T^\perp E_{\trusupp\setminus T}\nbeta_{\trusupp\setminus T}, \neps \rangle}{(n-j)\sigma^2} \\
    &  = \frac{\signalone(\trusupp,T)\|\Pi_T^\perp U\|^2}{n-j}+ \frac{2\sqrt{\signalone(\trusupp,T)} \langle \Pi_T^\perp U, \Pi_T^\perp U_\neps\rangle}{n-j} \\
    & =: B_1 + B_2\,,
\end{align*}
where $ U,U_\neps\sim\mathcal{N}(0,I_n)$ and $U\indep U_\neps$. 
\begin{align*}
    \frac{\|\Pi^\perp_T \neps\|^2/\sigma^2}{n-j} - \frac{\|\Pi^\perp_{\trusupp}\neps\|^2/\sigma^2}{n-\sps} &  = \frac{\|\Pi^\perp_{T} \neps\|^2/\sigma^2 - \|\Pi^\perp_{\trusupp} \neps\|^2/\sigma^2}{n-j} + \Big(\frac{1}{n-j} -\frac{1}{n-\sps} \Big)\|\Pi^\perp_{\trusupp} \neps\|^2/\sigma^2\\
    & =\frac{\|(\Pi_{\trusupp}-\Pi_{\trusupp\cap T})\neps\|^2 / \sigma^2}{n-j} - \frac{\|(\Pi_{T}-\Pi_{\trusupp\cap T})\neps\|^2/\sigma^2}{n-j} \\
    &+ \frac{j-\sps}{n-j}\times \frac{\|\Pi^\perp_{\trusupp} \neps\|^2/\sigma^2}{n-\sps} \\
    & \ge - \frac{\|(\Pi_{T}-\Pi_{\trusupp\cap T})\neps\|^2/\sigma^2}{n-j} 
    - \frac{\sps-j}{n-j}\times \frac{\|\Pi^\perp_{\trusupp} \neps\|^2/\sigma^2}{n-\sps}  \\
    & =: B_3 + B_4\,,
\end{align*}
Recall our short hand notation $\nu = \mineig \betam^2 /\sigma^2$, then
\begin{align*}
    & \prob\bigg[ \frac{\|\Pi^\perp_{T}Y\|^2/\sigma^2}{n-|T|} - \frac{\|\Pi^\perp_{\trusupp}Y\|^2/\sigma^2}{n-\sps} \le \frac{1}{4}\Big(\signalone(\trusupp,T) - \ell' \nu \Big)\bigg] \\
    \le & \prob\Big(B_1 \le \frac{1}{2}\signalone(\trusupp,T)\Big) + \sum_{k=2}^4 \prob\Big(B_k\le -\frac{1}{12}(\signalone(\trusupp,T)+\ell' \nu)\Big) \,.
\end{align*}
We deal with these 4 error probabilities individually. 

For $B_1$, analogous to the first part of proof of Lemma~\ref{lem:opt:ub:bss:gap:2}, we can conclude 
\begin{align*}
    \prob\Big(B_1 \le \frac{1}{2}\signalone(\trusupp,T)\Big) \le \exp(-(n-\ubsps) / 64)\,.
\end{align*}

For $B_2$, analogous to the second part of proof of Lemma~\ref{lem:opt:ub:bss:gap:2}, we firstly condition on $X_T$, then $2\langle \Pi_T^\perp U, \Pi_T^\perp U_\neps\rangle = W - \widetilde{W}$ where $W,\widetilde{W}\sim \chi^2_{n-j}$. Thus,
\begin{align*}
    & \quad \prob\Big(B_2\le -\frac{1}{12}(\signalone(\trusupp,T)+\ell' \nu)\Big)\\
    & = \prob\bigg(\frac{W-\widetilde{W}}{n-j}\le - \frac{1}{12}\frac{\ell' \nu + \signalone(\trusupp,T)}{\sqrt{\signalone(\trusupp,T)}}\bigg)\\
    & \le 2\prob\bigg(\frac{|W-(n-j)|}{n-j}\le \frac{1}{24}\frac{\ell' \nu + \signalone(\trusupp,T)}{\sqrt{\signalone(\trusupp,T)}}\bigg) \\
    & \le 2\exp\bigg(-(n-\ubsps)\min\Big(\frac{\ell' \nu + \signalone(\trusupp,T)}{\sqrt{\signalone(\trusupp,T)}}, \frac{(\ell' \nu + \signalone(\trusupp,T))^2}{\signalone(\trusupp,T)} \Big) / 9216\bigg) \\
    & \le  2\exp\bigg(-(n-\ubsps)\min\Big(\ell' \nu + \signalone(\trusupp,T), \sqrt{\ell' \nu + \signalone(\trusupp,T)} \Big) / 9216\bigg)\,.
\end{align*}
The last inequality is because 
\begin{align*}
    \frac{(\ell' \nu + \signalone(\trusupp,T))^2}{\signalone(\trusupp,T)} & = \frac{(\ell' \nu)^2 + \signalone^2(\trusupp,T) + 2\ell' \nu \signalone(\trusupp,T)}{\signalone(\trusupp,T)}\\
    & =  \signalone(\trusupp,T) + 2\ell' \nu + \frac{(\ell' \nu)^2}{\signalone(\trusupp,T)}  \\
    & \ge \signalone(\trusupp,T) + \ell'\nu \,.
\end{align*}

For $B_3$, we first condition on $X_T \AND X_{\trusupp}$, $\|(\Pi_{T}-\Pi_{\trusupp\cap T})\neps\|^2/\sigma^2 = Z \sim \chi^2_{|T\setminus \trusupp|}$. Then
\begin{align*}
    \prob\Big(B_3\le -\frac{1}{12}(\signalone(\trusupp,T)+\ell' \nu)\Big) & = \prob\Big(\frac{Z}{n-j}\ge \frac{1}{12}(\signalone(\trusupp,T)+\ell' \nu) \Big) \\
    & = \prob\Big(\frac{\chi^2_{| T\setminus \trusupp|}}{| T\setminus \trusupp|} - 1\ge \frac{(n-j)(\signalone(\trusupp,T)+\ell' \nu)}{12|T\setminus \trusupp|} -1 \Big) \\
    & \le \exp\bigg(-(n-j)\frac{\signalone(\trusupp,T) + \ell' \nu}{48} + \frac{|T\setminus \trusupp|}{4}\bigg) \\
    & \le \exp\bigg(-(n-\ubsps)\frac{(|\trusupp\setminus T| + \ell') \nu}{48} + \frac{|T\setminus \trusupp|}{4}\bigg) \,.
\end{align*}
The second to the last inequality holds when 
\begin{align*}
    \frac{(n-\ubsps)(\signalone(\trusupp,T)+\ell' \nu)}{48|T\setminus \trusupp|} -\frac{1}{4} \ge 1 \Leftarrow n-j\ge \frac{96|T\setminus \trusupp|}{(|\trusupp\setminus T|+\ell')\nu}\,,
\end{align*}
which is ensured by $n-\ubsps \ge 96/\nu$ because $|\trusupp\setminus T| + \ell' \ge |T\setminus \trusupp|$ by definition of $\ell'$.

For $B_4$, when $j>\sps$, $B_4\ge 0$. When $\sps > j$, we first condition on $X_{\trusupp}$, $\|\Pi^\perp_{\trusupp}\neps\|^2/\sigma^2 = \widetilde{Z}\sim \chi^2_{n-\sps}$. Then
\begin{align*}
    \prob\Big(B_4\le -\frac{1}{12}(\signalone(\trusupp,T)+\ell' \nu)\Big) & =  \prob\bigg( \frac{\widetilde{Z}}{n-\sps} \ge \frac{(n-j)(\signalone(\trusupp,T)+\ell'\nu)}{12(\sps-j)} \bigg) \\
    & = \prob\bigg( \frac{\chi^2_{n-\sps}}{n-\sps} - 1 \ge \frac{(n-j)(\signalone(\trusupp,T)+\ell'\nu)}{12(\sps-j)} -1 \bigg) \\
    & \le \exp\bigg(-(n-\sps)\Big(\frac{\signalone(\trusupp,T) + \ell'\nu}{48} \times \frac{n-j}{\sps-j} - \frac{1}{4}\Big)\bigg) \\
    & \le \exp\bigg(-(n-\sps)\frac{\signalone(\trusupp,T) + \ell'\nu}{48}\bigg) \\
    & \le \exp\bigg(-(n-\ubsps)\frac{ (|\trusupp\setminus T|+ \ell')\nu}{48}\bigg) \,.
\end{align*}
The first inequality holds when
\begin{align*}
    \frac{(n-j)(\signalone(\trusupp,T)+\ell'\nu)}{48(\sps-j)} -\frac{1}{4} \ge 1 \Leftarrow n-j\ge \frac{96(\sps-j)}{(|\trusupp\setminus T|+\ell')\nu}\,,
\end{align*}
which is ensure by $n-\ubsps\ge 96/\nu$ since $|\trusupp\setminus T| \ge \sps-j$. The second inequality holds because
\begin{align*}
    \frac{(n-j)(\signalone(\trusupp,T)+\ell'\nu)}{48(\sps-j)} -\frac{1}{4} &\ge \frac{(|\trusupp\setminus T| + \ell')\nu}{48} \times \frac{n-j}{\sps-j} -\frac{1}{4} \\
    & = \frac{(|\trusupp\setminus T| + \ell')\nu}{48} \bigg(\frac{n-j}{\sps-j} -\frac{12}{(|\trusupp\setminus T| + \ell')\nu}\bigg) \\
    & \ge \frac{(|\trusupp\setminus T| + \ell')\nu}{48}\,.
\end{align*}
The last inequality in equation above holds when
\begin{align*}
    \frac{n-j}{\sps-j} -\frac{12}{(|\trusupp\setminus T| + \ell')\nu} \ge 1 \Leftrightarrow n-\sps \ge \frac{12(\sps-j)}{(|\trusupp\setminus T|+\ell')\nu}\,,
\end{align*}
which is ensured by $n-\ubsps\ge 96/\nu$.

Finally, combining these error probability bounds, we conclude
\begin{align*}
    & \prob\bigg[ \frac{\|\Pi^\perp_{T}Y\|^2/\sigma^2}{n-|T|} - \frac{\|\Pi^\perp_{\trusupp}Y\|^2/\sigma^2}{n-\sps} \le \frac{1}{4}\Big(\signalone(\trusupp,T) - \ell' \nu \Big)\bigg] \\
    \le & \prob\Big(B_1 \le \frac{1}{2}\signalone(\trusupp,T)\Big) + \sum_{k=2}^4 \prob\Big(B_k\le -\frac{1}{12}(\signalone(\trusupp,T)+\ell' \nu)\Big) \\
    \le& 5\exp\bigg(-(n-\ubsps)\frac{\min\Big((|\trusupp\setminus T|+\ell')\nu,1\Big)}{9216} + \frac{|T\setminus \trusupp|}{4}\bigg)\,.
\end{align*}
\end{proof}

\section{Proof of Theorem~\ref{thm:opt:bss:lb1:unknown}}\label{app:opt:bss:lb1:unknown}
\begin{proof}
    Again, we consider the covariance matrix of $X$:
    \begin{align*}
        \Sigma = (1-\rho)I_\dd + \rho \mathbf{1}_\dd\mathbf{1}_\dd\T
    \end{align*}
    with $\rho = 1-\mineig$. Then for any $T\in\suppspaceub$ with $|T|=j$, and $|S\setminus T| = r$, we can calculate the conditional covariance matrix
    \begin{align*}
        \Sigma_{S\setminus T\given T} & = \Sigma_{(S\setminus T)(S\setminus T)} - \Sigma_{(S\setminus T)T}\Sigma^{-1}_{TT}\Sigma_{T(S\setminus T)} \\
        & =  (1-\rho)\bigg(I_r + \frac{\rho}{1-\rho + j\rho} \mathbf{1}_r\mathbf{1}_r\T\bigg) \,,
    \end{align*}
    then the minimum eigenvalue is
    \begin{align*}
        \lambda_{\min}(\Sigma_{S\setminus T\given T}) = \begin{cases}
        (1 - \rho) \times (1+\frac{\rho}{1-\rho+j\rho}) & r = 1 \\
        1-\rho & r\ge 2
        \end{cases} \,.
    \end{align*}
    Since $\lambda_{\min}(\Sigma_{S\setminus T\given T})$ is independent with the choice of $(S,T)$, this covariance matrix $\Sigma$ satisfies the requirement on $\Sigmaspace_\dd^{\ubsps} (\mineig)$:
    \begin{align*}
        \min_{S\in\suppspaceub}\min_{\substack{T\in \suppspaceub\setminus\{S\}\\ T\not\supseteq S}} \lambda_{\min}(\Sigma_{S\setminus T\given T}) = 1 - \rho = \mineig \,.
    \end{align*}
    Note that we only take $T\not\supseteq S$ to make sure $r\ge 1$.
    Now we fix the covariance matrix $\Sigma$, consider the ensemble with support size one: Each integer $k\in[\dd]$ determines a model $Y = X_k \betam +\neps$. Thus the cardinality of this model ensemble is $|\binom{\dd}{1}| = \dd$. Now we calculate the KL divergence between two models specified by $k$ and $j$. We further denote and the models determined by them to be $P_k$ and $P_j$. Therefore,
    \begin{align*}
        \mathbf{KL}(P_k\| P_j) & = \E_{P_k} \log \frac{P_k}{P_j} \\
        & = \E_X (X_k-X_j)^2\betam^2 / 2\sigma^2 \\ 
        & = 2(1-\rho)\times \frac{\betam^2}{2\sigma^2} \\
        & =  \frac{\mineig\betam^2}{\sigma^2}
    \end{align*}
    Thus $\mathbf{KL}(P_k\|P_j) \le \betam^2\mineig/\sigma^2$, which holds for any pair of $j,k\in [\dd]$ and leads to a upper bound for any two models in this ensemble.
    Finally, we apply Fano's inequality Corollary~\ref{coro:fano} with KL divergence upper bound $\betam^2\mineig/\sigma^2$ and ensemble cardinality $\dd$, which completes the proof.
\end{proof}

\section{Proof of Lemma~\ref{thm:poly:main}}\label{app:poly:main}
\begin{proof}
    Given any polynomial time support estimator $\estsupp=\estsupp(X,Y)$, we construct an estimator for $\nbeta$ vector as follows:
    \begin{enumerate}
        \item Split the data $(Y,X)$  into two folds with equal size $(Y^{(1)},X^{(1)})$ and $(Y^{(2)},X^{(2)})$;
        \item Estimate support using the first fold $\estsupp = \estsupp(Y^{(1)},X^{(1)})$;
        \item Estimate the $\nbeta$ vector by
        \begin{align*}
            \widehat{\nbeta} = \begin{pmatrix}
                \widehat{\nbeta}_{\estsupp} \\ \widehat{\nbeta}_{\estsupp^c}
            \end{pmatrix}
            = \begin{pmatrix}
                 ({X^{(2)}}\T X^{(2)})^{-1}{X^{(2)}}\T Y^{(2)} \\ \mathbf{0}_{\dd-\sps}
            \end{pmatrix}      \,.
        \end{align*}
    \end{enumerate}
    Therefore, $\widehat{\nbeta}$ is a polynomial time estimator for $\nbeta$.
    We are going to employ the construction in the following lemma.
    \begin{lemma}[Theorem~1, \citet{zhang2014lower}]
        If $\mathbf{NP}\not\subset \mathbf{P}\backslash \mathbf{poly}$, then for any $\delta\in  (0,1)$, any $b\in \mathbb{Z}_+$, any polynomial functions $G:(\mathbb{Z}_+)^3 \to \mathbb{R}_+$ and $F,H: \mathbb{Z}_+ \to \mathbb{R}_+$, there exists a sparsity level $\sps\ge 1$ such that for any $\dd\in [4\sps, F(\sps)]$, $n'\in [c_1 \sps\log \dd, F(\sps)]$, and $\gamma\in [2^{-G(n,\dd,\sps)}, 1/{24\sqrt{2}})$, there exists a design matrix $\widetilde{X}\in\mathbb{R}^{n'\times \dd}$ such that:
        \begin{enumerate}
            \item The RE constant $|\gamma(\widetilde{X})-\gamma|\le 2^{-G(n,\dd,\sps)}$;
            \item For any $(b,G,H)$-efficient estimator $\widehat{\nbeta}$ with knowledge of $\sps$, the mean-squared prediction risk is lower bounded as
            \begin{align*}
                \max_{\nbeta\in\betaspace_{\dd,\sps}}\E \frac{\|\widetilde{X}(\widehat{\nbeta}-\nbeta)\|^2}{n'} \ge \frac{c_2}{\gamma^2}\frac{\sigma^2 \sps^{1-\delta}\log\dd }{n'} \,.
            \end{align*}
        \end{enumerate}
    \end{lemma}
    With the same $\delta$, polynomial functions $F,G,H$ stated in the theorem, there exists sparsity level $\sps\ge 1$ such that for the $\dd,n/2,\gamma$ satisfying the requirement, we have $\widetilde{X}\in\mathbb{R}^{(n/2)\times \dd}$ such that $|\gamma(\widetilde{X}) - \gamma | \le 2^{-G(n,\dd,\sps)}$ and for any $(b,G,H)$-efficient estimator $\widetilde{\nbeta}$ with knowledge of $\sps$,
    \begin{align*}
        \max_{\nbeta\in\betaspace_{\dd,\sps}} \E \frac{\|\widetilde{X}(\widetilde{\nbeta}-\nbeta)\|^2}{n/2} \ge \frac{C'}{\gamma^2}\frac{\sigma^2 \sps^{1-\delta}\log\dd}{n/2}\,.
    \end{align*}
    We now construct $X = (\widetilde{X}\T , \widetilde{X}\T )\T$ by stacking two copies of $\widetilde{X}$ and take the corresponding maximizer $\nbeta$ to form $Y = (Y^{(1)},Y^{(2)}) = X\nbeta + \neps$ with $\neps = (\neps^{(1)}, \neps^{(2)})$. 
    Note that $\gamma(X) = \gamma(\widetilde{X})$ by definition, thus $|\gamma(X)-\gamma|=|\gamma(\widetilde{X})-\gamma|\le 2^{-G(\sps)}$.
    
    We then analyze the property of $\widehat{\nbeta}$ on this construction. Note that $\neps^{(1)}\indep \neps^{(2)}$, and $\estsupp$ only depends on $\neps^{(1)}$ via $Y^{(1)}$, thus $\estsupp\indep \neps^{(2)}$. 
    Denote that $\widetilde{\Pi}_T = \widetilde{X}(\widetilde{X}\T \widetilde{X})^{-1}\widetilde{X}\T$, and $\widetilde{\Pi}^\perp_T = I_{n/2} - \widetilde{\Pi}_T$. We have
    \begin{align}\label{eq:poly:ineq}
    \begin{aligned}
        \frac{C'}{\gamma^2}\frac{\sigma^2 \sps^{1-\delta}\log\dd}{n/2} & \le \E \frac{\|X(\widehat{\nbeta}-\nbeta)\|^2}{n} 
        = \E \frac{\|\widetilde{X}(\widehat{\nbeta}-\nbeta)\|^2}{n/2} \\
        & = \frac{1}{n/2} \E \Big\|\widetilde{X}_{\estsupp}(\widetilde{X}_{\estsupp}\T \widetilde{X}_{\estsupp})^{-1}\widetilde{X}_{\estsupp}\T (\widetilde{X}_{\trusupp}\nbeta_{\trusupp} + \neps^{(2)}) - \widetilde{X}_{\trusupp}\nbeta_{\trusupp} \Big\|^2 \\
        & = \frac{1}{n/2} \E \bigg[\Big\|\widetilde{X}_{\estsupp}(\widetilde{X}_{\estsupp}\T \widetilde{X}_{\estsupp})^{-1}\widetilde{X}_{\estsupp}\T (\widetilde{X}_{\trusupp}\nbeta_{\trusupp} + \neps^{(2)}) - \widetilde{X}_{\trusupp}\nbeta_{\trusupp} \Big\|^2 \given \estsupp=\trusupp \bigg] \prob(\estsupp=\trusupp) \\
        & \quad + \frac{1}{n/2} \E \bigg[\Big\|\widetilde{X}_{\estsupp}(\widetilde{X}_{\estsupp}\T \widetilde{X}_{\estsupp})^{-1}\widetilde{X}_{\estsupp}\T (\widetilde{X}_{\trusupp}\nbeta_{\trusupp} + \neps^{(2)}) - \widetilde{X}_{\trusupp}\nbeta_{\trusupp} \Big\|^2 \given \estsupp\ne \trusupp \bigg] \prob(\estsupp\ne \trusupp) \\
        & = \frac{1}{n/2} \E \bigg[\Big\|\widetilde{\Pi}_{\trusupp} \neps^{(2)} \Big\|^2 \given \estsupp=\trusupp \bigg] \prob(\estsupp=\trusupp) \\
        & \quad + \frac{1}{n/2} \E \bigg[\Big\| \widetilde{\Pi}_{\trusupp} \neps^{(2)} - \widetilde{\Pi}^\perp_{\estsupp}\widetilde{X}_{\trusupp\setminus\estsupp}\nbeta_{\trusupp\setminus\estsupp} \Big\|^2 \given \estsupp\ne \trusupp \bigg] \prob(\estsupp\ne \trusupp) \\
        & = \frac{\sps\sigma^2}{n/2}\prob(\estsupp=\trusupp) + \E \bigg[ \frac{\sps\sigma^2}{n/2} + \frac{\|\widetilde{\Pi}^\perp_{\estsupp}\widetilde{X}_{\trusupp\setminus\estsupp}\nbeta_{\trusupp\setminus\estsupp}\|^2}{n/2}  \given \estsupp\ne \trusupp \bigg] \prob(\estsupp\ne \trusupp) \\
        & \le 2\times \frac{\sps\sigma^2}{n/2} + \prob(\estsupp\ne \trusupp) \times \max_{T\ne \trusupp} \frac{\|\widetilde{\Pi}^\perp_{T}\widetilde{X}_{\trusupp\setminus T}\nbeta_{\trusupp\setminus T}\|^2}{n/2} \\
        & = 2\times \frac{\sps\sigma^2}{n/2} + \prob(\estsupp\ne \trusupp) \times \max_{T\ne \trusupp} \frac{\|\Pi^\perp_{T} X_{\trusupp\setminus T}\nbeta_{\trusupp\setminus T}\|^2}{n} \,.
        \end{aligned}
    \end{align}
    Recall that $\Pi^\perp_T = I_n - X_T(X_T\T X_T)^{-1}X_T\T$. 
    Since $\gamma < \sps^{-\delta/2}$, then $\frac{1}{\gamma^2}> \sps^\delta$, and
    \begin{align*}
        \frac{C'}{\gamma^2}\frac{\sigma^2\sps^{1-\delta}\log\dd}{n/2} > \frac{C'\sigma^2\sps\log\dd}{n/2} \gtrsim \frac{2\sigma^2\sps}{n/2}\,,
    \end{align*}
    for sufficient large $\dd$. Therefore, for the inequality~\eqref{eq:poly:ineq} to hold, we must have
    \begin{align*}
        \prob(\estsupp\ne \trusupp) \times \max_{T\ne \trusupp} \frac{\|\Pi^\perp_{T} X_{\trusupp\setminus T}\nbeta_{\trusupp\setminus T}\|^2}{n} \ge \frac{C_2}{\gamma^2}\frac{\sigma^2 \sps^{1-\delta}\log\dd}{n/2} 
    \end{align*}
    for some constant $C_2$. Moving the signal term to the right hand side completes the proof.
\end{proof}

\section{Auxiliary lemmas}
We will employ the tail probability bounds for $\chi^2$ distribution \citep{laurent2000adaptive}.
\begin{lemma}\label{lem:chisq:orig}
If $Z\sim \chi^2_m$ with degree $m$, then for any $t\ge 0$,
\begin{align*}
    & \prob\bigg[\frac{Z-m}{m} \ge 2(\sqrt{t}+t)\bigg] \le \exp(-mt) \\
    & \prob\bigg[\frac{Z-m}{m} \le -2\sqrt{t}\bigg] \le \exp(-mt) \,.
\end{align*}
\end{lemma}
\noindent
Especially, we work with the following concentration bounds.
\begin{lemma}\label{lem:chisq}
If $Z\sim \chi^2_m$ with degree $m$, then for any $t\ge 0$,
\begin{align*}
    \prob\bigg[\frac{|Z-m|}{m} \ge 4t\bigg] \le \exp(-m\min(t,t^2)) \,.
\end{align*}
\end{lemma}
\begin{proof}
If $t \ge 1$, then $2(\sqrt{t} + t)\le 4t$, $-4t \le -2t \le -2\sqrt{t}$, thus
\begin{align*}
    & \prob\bigg[\frac{Z-m}{m} \ge 4t\bigg] \le \prob\bigg[\frac{Z-m}{m} \ge 2(\sqrt{t}+t)\bigg] \le \exp(-mt) \\
    & \prob\bigg[\frac{Z-m}{m} \le -4t\bigg] \le \prob\bigg[\frac{Z-m}{m} \le -2\sqrt{t}\bigg] \le \exp(-mt) \,.
\end{align*}
If $t\in[0,1)$, let $h=t^2 \in [0,1)$, then $2(\sqrt{h}+h) \le 4\sqrt{h}$, $-4\sqrt{h}\le -2\sqrt{h}$, thus
\begin{align*}
    & \prob\bigg[\frac{Z-m}{m} \ge 4t\bigg] = \prob\bigg[\frac{Z-m}{m} \ge 4\sqrt{h} \bigg]\le \prob\bigg[\frac{Z-m}{m} \ge 2(\sqrt{h}+h)\bigg] \le \exp(-mh) = \exp(-mt^2) \\
    & \prob\bigg[\frac{Z-m}{m} \le -4t\bigg]=\prob\bigg[\frac{Z-m}{m} \le -4\sqrt{h}\bigg] \le \prob\bigg[\frac{Z-m}{m} \le -2\sqrt{h}\bigg] \le \exp(-mh) = \exp(-mt^2) \,.
\end{align*}
\end{proof}
\noindent
For lower bound techniques, we mainly apply the Fano's inequality.
\begin{lemma}[\citet{yu1997assouad}, Lemma~3]\label{lem:fano}
For a model family $\mclass$ contains $M$ many distributions indexed by $j=1,2,\ldots,M$ such that 
\begin{align*}
\alpha &= \max_{P_j\ne P_k\in\mclass}\mathbf{KL}(P_j\| P_k) \\ 
s& =\min_{P_j\ne P_k\in\mclass}\mathbf{dist}(\theta(P_j),\theta(P_k)) \,,   
\end{align*}
where $\theta$ is a functional of its distribution argument. Then for any estimator $\widehat{\theta}$ for $\theta(P)$,
\begin{align*}
    \inf_{\widehat{\theta}}\sup_{P\in\mclass} \E_P \mathbf{dist}(\theta(P), \widehat{\theta}) \ge \frac{s}{2}\bigg(1 - \frac{\alpha + \log 2 }{\log M}\bigg) \,.   
\end{align*}
\end{lemma}
\noindent
Set $\theta(P_j)=j$ to be the index, $\mathbf{dist}(\cdot,\cdot) = \mathbf{1}\{\cdot \ne \cdot\}$, consider $P_j$ to be a product measure of $n$ i.i.d. samples for any $P_j\in \mclass$, then Lemma~\ref{lem:fano} under model selection context can be stated as follows:
\begin{corollary}[Fano's inequality]\label{coro:fano}
For a model family $\mclass$ contains $M$ many distributions indexed by $j=1,2,\ldots,M$ such that 
$\alpha= \max_{P_j\ne P_k\in\mclass}\mathbf{KL}(P_j\| P_k)$. 
If the sample size is bounded as
\begin{align*}
    n \le \frac{(1-2\delta)\log M}{\alpha} \,,
\end{align*}
then for any estimator $\widehat{\theta}$ for the model index:
\begin{align*}
    \inf_{\widehat{\theta}}\sup_{j\in[M]} P_j(\widehat{\theta}\ne j) \ge \delta - \frac{\log 2}{\log M} \,.
\end{align*}
\end{corollary}

\bibliographystyle{abbrvnat} 
\bibliography{foo}      

\end{document}